\documentclass[12pt]{amsart}

\newtheorem{theorem}{Theorem}[section]
\newtheorem{lemma}[theorem]{Lemma}

\newtheorem{proposition}[theorem]{Proposition}

\usepackage[bookmarks]{hyperref}
\usepackage{cite}
\usepackage[makeroom]{cancel}
\usepackage{graphicx,amscd,amsthm,amsfonts,amsopn,amssymb,verbatim,enumerate,xcolor}

\usepackage[letterpaper,left=1in,top=1in,right=1in,bottom=1in]{geometry}
\usepackage{xcolor}

\def\nint{\mathop{\diagup\kern-13.0pt\int}}

\def\bas{\begin{align*}}
\def\eas{\end{align*}}
\def\bi{\begin{itemize}}
\def\ei{\end{itemize}}

\def\emph#1{{\it #1}}

\theoremstyle{definition}
\newtheorem{remark}[theorem]{Remark}

\numberwithin{equation}{section}
\title[A conditional clock barrier at the $C^{1,\frac{1}{3}}$ threshold for axisymmetric Euler]{A conditional Lagrangian clock barrier at the $C^{1,\frac{1}{3}}$ threshold for axisymmetric Euler without swirl}
\author{Ovidiu-Neculai Avadanei}
\begin{document}
\begin{abstract}
   We consider axisymmetric no-swirl solutions to the three-dimensional incompressible Euler equations, with initial velocity in $C^{1,\alpha}\cap L^2$, where $\alpha\in\left[\frac{1}{3},1\right)$. In a major breakthrough, Shkoller introduced a clock-and-driver framework that he used in order to prove finite-time type I blow-up below the $C^{1,\frac{1}{3}}$ threshold in this setting. Motivated by this, we define Lagrangian classes of coherent conditional solutions for which the same mechanism yields a supercritical-critical barrier to blow-up when $\alpha\geq\frac{1}{3}$. When $\alpha>\frac{1}{3}$, the aforementioned barrier is genuinely depleted, whereas at the critical endpoint $\alpha=\frac{1}{3}$, we obtain an exponential bound preventing blow-up. In the general case, we formulate a matrix-clock criterion in terms of the smallest singular value of the deformation gradient and show that, under transverse cusp-tail, longitudinal, off-clock, Dini, and suitable geometric coherence hypotheses, this singular value cannot collapse in finite time. In particular, we also show that these classes of coherent solutions include the classical $C^{1,\alpha}\cap L^2$ ones locally in time. In the on-axis case, the criterion reduces to the scalar clock inequality $\displaystyle \dot{J}(t)\gtrsim -B(t)J(t)-CJ(t)^{3\alpha}$, which rules out Shkoller-type clock collapse for $\alpha\geq\frac{1}{3}$. These results do not enlarge the known Lorentz-space global regularity classes. Rather, they in particular identify the supercritical Lagrangian obstruction dual to Shkoller's subcritical blow-up mechanism in the case $\alpha>\frac{1}{3}$. 
\end{abstract}
\maketitle

\section{Introduction}
In this paper we are concerned with the incompressible Euler equations in three dimensions:
\begin{equation}\label{Incompressible Euler}
\begin{aligned}
   u_t+(u\cdot\nabla)u+\nabla p&=0\text{ in }[0,T)\times \mathbb{R}^3\\
   \nabla\cdot u&=0\text{ in }[0,T)\times \mathbb{R}^3\\
   u(0,\cdot)&=u_0.
   \end{aligned}
\end{equation}
It is known that the vorticity $\omega=\nabla\times u$ satisfies the equation
\begin{equation}\label{Vorticity Equation}
    \omega_t+(u\cdot\nabla)\omega=(\omega\cdot \nabla)u,
\end{equation}
where $u$ and $\omega$ are related by the Biot-Savart Law
\begin{equation}\label{Biot-Savart}
    u(t,x)=-\frac{1}{4\pi}\int_{\mathbb{R}^3}\frac{x-y}{|x-y|^3}\times\omega(t,y)\,dy
\end{equation}

In particular, the vorticity equation \eqref{Vorticity Equation} shows that any potential blow-up phenomenon is caused by the stretching term $(\omega\cdot \nabla)u$. Moreover, the Beale-Kato-Majda criterion \cite{2} also states that in order for blow-up to occur at a finite time $T_*$, one must have $\int_0^{T_*}\|\omega(t)\|_{L^\infty}\,dt=\infty$.

Throughout this paper, we shall work with axisymmetric solutions with no-swirl, denoted by $u$, for which, in cylindrical coordinates, one can write $u=u_re_r+u_ze_z$, and $\omega=\omega_\theta e_\theta$. In this setting if one defines $D_t:=\partial_t+u\cdot\nabla$, which is also known as the material derivative, it can be immediately seen that the \textbf{specific vorticity} $\frac{\omega^\theta}{r}$ satisfies $\displaystyle D_t\left(\frac{\omega^\theta}{r}\right)=0$, i.e. it is transported.
We shall always assume that the initial data $u_0$ satisfies $u_0\in C^{1,\alpha}\cap L^2$, $1>\alpha\geq\frac{1}{3}$.

Smooth axisymmetric solutions with no-swirl are globally regular under the classical assumptions, and the transported specific vorticity $\frac{\omega^\theta}{r}$ is central to the known regularity theory \cite{28, 22, 25}. Thus, possible singularity formation in this class is naturally tied to low regularity or borderline control of the stretching mechanism. Moreover, it turns out that the stretching mechanism can generate long-time growth and outward migration phenomena even at high regularity (see \cite{16} and the references therein).

In a major breakthrough, Shkoller \cite{Shkoller} proved finite-time singularity formation for the axisymmetric no-swirl Euler equations in $\mathbb{R}^3$ with velocity in $C^{1,\alpha}$, where $\alpha\in\left(0,\frac{1}{3}\right)$.

To be more specific, he constructed a target blow-up profile, for which he proved type I blow-up, and then proved its stability under an admissible class of perturbations. This was the first result pertaining to the question of $C^{1,\alpha}$ singularity formation that was proved in a Lagrangian setting.

One of his ideas was to introduce a Lagrangian clock and driver framework, modeled by the collapse-clock and stagnation point equations
\begin{equation}
  \begin{aligned}
\frac{d}{dt}J(t)&=\frac{1}{2}W(t)J(t)\\
\frac{d}{dt}W(t)&=-\frac{1}{2}W(t)^2-\Pi_0(t)
\end{aligned} 
\end{equation}
Here, $W(t)=\partial_zu^z(0,0,t)$ is the axial strain, and $J(t)$ is the meridional Jacobian associated to the fluid flow (for a precise definition, see Theorem \ref{Main Result 2}), while $\Pi_0(t)$ is the non-local principal value contribution of the on-axis pressure Hessian. His idea was to construct a solution and a class of admissible perturbations with the property that $J(t)$ tends to zero in finite time, which is equivalent to singularity formation, and that the pressure strain $\Pi_0(t)$, which is the only obstruction that could have prevented blow-up, is completely dominated by the Riccati term $W(t)^2$.

Another ingredient in his paper was to introduce a kinematic drift law which also explains the $\alpha=\frac{1}{3}$ threshold. In a nutshell, the idea is that the particles which drive the on-axis strain at late times can only originate from ones whose Lagrangian polar angle does not exceed $J(t)^3$. Together with the H\" older cusp vanishing property of the target profile that he constructed, this implies that the driver sector is depleted by a factor of $J(t)^{3\alpha}$, a mechanism that he calls drift-induced depletion. Heuristically, $W(t)\simeq -\Gamma J(t)^{3\alpha-1}$, hence $\frac{d}{dt}J(t)\simeq -\Gamma J(t)^{3\alpha}$. The latter can have a finite-time zero only when $3\alpha<1$, which explains the threshold.

Another important idea was to prove that the ratio between the Riccati term $\frac{1}{2}W(t)^2$ and the pressure strain $\Pi_0(t)$ is greater or equal than $1$ in the model case (the blow-up profile), which means that the latter cannot prevent singularity formation.

In this paper, our goal is to turn around Shkoller's clock and driver Lagrangian framework and prove that in the context of an admissible class of conditional coherent solutions, blow-up cannot occur when $\alpha\geq\frac{1}{3}$. Our results are mechanism proofs, and show that Shkoller's framework not only explains how blow-up can occur when $\alpha<\frac{1}{3}$, but that it also explains why it cannot happen in the first place when $\alpha\geq\frac{1}{3}$. 

We shall work in cylindrical coordinates. Let $(R,Z)$ be the initial cylindrical coordinates of a particle in the poloidal $(r,z)$- plane at the initial time $t=0$, and let $\Theta\in[0,2\pi)$ denote its azimuthal angle. Let $\Phi_t$ be its flow map defined by
\begin{align*}
    \partial_t\Phi_t(R,\Theta,Z)&=u(t,\Phi_t(R,\Theta,Z))\\
    \Phi_0(R,\Theta,Z)&=(R,\Theta,Z).
\end{align*}

Given that the flow is axisymmetric without swirl, the azimuthal angle is conserved, in the sense that $\theta(t)\equiv\Theta$ for every $t$. There also exists a \textbf{poloidal flow map}
\begin{align*}
    \Phi(R,Z,t)=(\varphi^r(R,Z,t),\varphi^z(R,Z,t)),
\end{align*}
where
\begin{align*}
    \frac{d}{dt}(\varphi^r(R,Z,t),\varphi^z(R,Z,t))&=(u^r(\varphi^r(R,Z,t),\varphi^z(R,Z,t)),u^z(\varphi^r(R,Z,t),\varphi^z(R,Z,t)))\\
    (\varphi^r(R,Z,0),\varphi^z(R,Z,0))&=(R,Z).
\end{align*}
The flow $\Phi_t$ will also be given by
\begin{align*}
    \Phi_t(R,\Theta,Z)&=(\varphi^r(R,Z,t)\cos\Theta, \varphi^r(R,Z,t)\sin\Theta,\varphi^z(R,Z,t)).
\end{align*}
The main result is the following
\begin{theorem}\label{Main Result}
    Let $\alpha\in\left[\frac{1}{3},1\right)$, and $u$ be an axisymmetric no-swirl solution on $[0,T_*)$ with initial data $u_0\in C^{1,\alpha}(\mathbb{R}^3)\cap L^2(\mathbb{R}^3)$. Assume that $u$ belongs to $\mathcal{S}^\alpha_{\rm general}([0,T_*))$, where the class of conditional coherent solutions $\mathcal{S}^\alpha_{\rm general}([0,T_*))$ shall be defined in section \ref{Section 2}, and the $\mathcal{S}^\alpha_{\rm general}([0,T_*))$ constants remain finite, in the sense that one has uniform control over the transverse cusp-tail, longitudinal, off-clock, Dini, and suitable geometric coherence properties. Then, $T_*$ is not a blow-up time.
\end{theorem}
Theorem \ref{Main Result} is a \textbf{matrix-clock criterion}: it rules out collapse of the smallest singular value of $D\Phi_t$, and hence precludes vorticity blow-up through the Cauchy formula.

Our second result can be stated as follows
\begin{theorem}\label{Main Result 2}
    Let $\alpha\in\left[\frac{1}{3},1\right)$, and $u$ be a solution in $ \mathcal{S}^\alpha_{\rm vert}([0,T))$ with initial data $u_0\in C^{1,\alpha}(\mathbb{R}^3)\cap L^2(\mathbb{R}^3)$ and is odd in $z$, where the class of conditional coherent solutions $\mathcal{S}^\alpha_{\rm vert}$ shall be defined in section \ref{Section 2}. Assume that $T$ is the maximal time of existence of the solution $u\in\mathcal{S}^\alpha_{\rm vert}([0,T))$, and that the $\mathcal{S}^\alpha_{\rm vert}([0,T))$ constants remain finite, in the sense that one has uniform control over the integrated cusp-tail, off-clock, and suitable geometric coherence properties. Let $a_*=(0,\Theta,Z)$ be an arbitrary axis label, and we consider the reduced meridional Jacobian given by
\begin{align*}
    J_*(t)=\lim_{\substack{R\rightarrow 0}}\det D_{R,Z}(\varphi^r(R,Z,t),\varphi^z(R,Z,t)).
\end{align*} 
If $T<\infty$, then there exists a positive constant $c_T$ independent of $a_*$ such that $J_*(t)\geq c_T$ for every $t\in[0,T)$.
\end{theorem}
Theorem \ref{Main Result 2} is the on-axis version of Theorem \ref{Main Result}. In this setting, the principal value obstruction disappears due to the fact that the angular vorticity vanishes on the axis and the transverse singular values are balanced by symmetry. This also provides us with the setting in which Shkoller's fundamental clock-barrier model is the most visible.

We must also mention that the classes of solutions $ \mathcal{S}^\alpha$ should be interpreted as coherent-clock classes and not as replacements for the Lorentz-space global regularity theory. In Section \ref{Section 2}, we shall describe the geometric assumptions which allow us to carry out our aforementioned clock argument in reverse. This will allow us to conclude that a supercritical or critical solution satisfying these coherence assumptions cannot exhibit blow-up caused by a Shkoller-type mechanism.

\subsection{Historical references}
Prior to Shkoller's work, all of the known $C^{1,\alpha}$ constructions for the Euler equations relied on an Eulerian approach or were boundary-based in nature. In particular, they did not involve a Lagrangian clock-collapse analysis in the axisymmetric no-swirl whole space setting. All of these constructions were carried out in settings in which one either allows for rough vorticities, or makes use of a boundary geometry.

In a major breakthrough, Elgindi \cite{13} proved that finite-time singularity happens in $C^{1,\alpha}$, where $0<\alpha\ll1$. One crucial aspect is that his theory is perturbative in the H\" older exponent $\alpha$, and as long as this parameter is sufficiently small, the aforementioned pressure and strain coupling effects can be shown to remain small, thus allowing for singularity formation to occur. For other developments in the small $\alpha$ regime, such as stability and geometric refinements, see \cite{14,10,15} and references therein. For other developments on the competition between advection and vortex stretching, see \cite{29}.  In the same small $\alpha$ regime, a recent paper of Córdoba–Martínez-Zoroa–Zheng \cite{10} came up with a multi-region  architecture for solutions with $C^{1,\alpha}\cap L^2$ regularity that are smooth away from a point.

In the case of a physical boundary, Chen-Hou proved finite-time blow-up for axisymmetric Euler and other related models with $C^{1,\alpha}$ velocity in \cite{5}, and have also recently developed a computer-assisted framework in which blow-up occurs for smooth boundary data \cite{6,7}. For a more general overview and further references, we refer the reader to the survey of Drivas-Elgindi \cite{12}, while for the matter of globally self-similar finite-energy  scenarios for 3D Euler, including in the axisymmetric setting, see Constantin-Ignatova-Vicol \cite{9}.

At the endpoint and in related higher-dimensional settings, we also mention the recent work of Shao-Wei-Zhang \cite{Shao-Wei-Zhang}, who prove global regularity for axisymmetric no-swirl Euler solutions under weak critical assumptions, such as $\frac{\omega_0}{r^{d-2}}\in L^{\frac{d}{d-2},\infty}$, which in three dimensions corresponds to the weak $L^3$ endpoint scale.

For the absence of blow-up for axisymmetric $C^{1,\alpha}$ solutions when $\alpha>\frac{1}{3}$, see the rigidity result of Saint-Raymond \cite{24}, as well as the one of Danchin \cite{11}, which also provides the sharpest formulation, in which the specific vorticity $\xi=\frac{\omega^\theta}{r}$ is in assumed to belong to the $L^{3,1}$ Lorentz space, which is critical to the argument. Bang and Cheskidov have recently proved in \cite{Bang-Cheskidov} that the $L^{3,1}$ space is sharp, in the sense that for every $q>1$, there exists multi-ring data $\omega_0\in L^\infty$ with $\xi\in L^{3,q}$ for which the $L^\infty$ the vorticity exhibits norm inflation.

\subsection{A brief outline of the paper}
In section \ref{Section 2}, we introduce our classes of admissible solutions, for both general and on-axis drivers. These are spaces with some natural geometrical properties, which we also describe and motivate there. The reason why we also need general drivers is because our aim in Theorem \ref{Main Result} is to prove the absence of singularity formation, which means that we must show that we can control any general drivers, not only the on-axis ones.

In section \ref{Section 3}, we prove that our coherence classes are not empty. In fact, we show that for a sufficiently short time, any classical $C^{1,\alpha}\cap L^2$ local axisymmetric solution with no swirl of \eqref{Incompressible Euler} belongs to both coherence classes. 

In section \ref{Section 4}, we provide the proof of Theorem \ref{Main Result}. The idea here is to work in a matrix clock setting, which is a generalization of Shkoller's on-axis driver one, and to study the smallest singular value of the deformation gradient $D\Phi_t$, for which we also study the time evolution. In particular, this equation also motivates our choice for the driver. We then carry out the proof of Theorem \ref{Main Result} by assuming a suitable driver bound, whose proof is postponed until section \ref{Section 5}.

In section \ref{Section 5}, we provide the proof of the general driver estimate. We split our analysis into far-field and near-field labels. For the former, we close the estimate by a localized energy argument, whereas the analysis for near-field estimates is more subtle, given that in particular, we must make sure to avoid principal value singularities. Here we use a singular value decomposition for the deformation gradient $D\Phi_t$,  further split the space into cone and off-cone dyadic regions, and carry out our estimate in each of them.

In section \ref{Section 6}, we provide the proof of Theorem \ref{Main Result 2}. For this purpose, we restrict our attention to on-axis drivers. This is also the place where Shkoller's clock and driver mechanism is the most visible, as our previous analysis for general drivers can be carried out with significant simplifications. For example, the symmetry assumptions ensure that $\omega^\theta$ vanishes on the axis, so we do not have to worry about potential singularities in the principal-value integrals.

\section{The class of admissible data and solutions
}\label{Section 2}
As we have already seen, our initial data satisfies the condition $u_0\in C^{1,\alpha}\cap L^2$, is axisymmetric. In what follows, we shall describe the conditional coherent classes of solutions $\mathcal{S}^\alpha_{\rm general}$ and $\mathcal{S}^\alpha_{\rm vert}$.
\subsection{General labels} 

We shall now describe the class of coherent conditional solutions $\mathcal{S}^\alpha_{\rm general}$. We must mention that these conditions are merely assumed to be true as part of the definition of our class of admissible solutions, and we do not pursue the proof of their propagation in this paper. 

Let $\omega_0=\nabla\times u_0$, and $a_*$ be an arbitrary label. We also consider $\eta=(\eta_1,
\eta_2,\eta_3)$ and $\displaystyle \eta_\perp=(\eta_1,\eta_2)$.

We fix $t\in[0,T)$, and we assume that there is a measurable choice for the orthogonal matrix $V=V(a_\ast,t)$ in the singular value decomposition of $D_a\Phi_t(a_\ast)$, written as
\begin{align*}
    D_a\Phi_t(a_\ast)&=U\begin{pmatrix}\tau_1 & 0 & 0\\0 & \tau_2 & 0\\0& 0& \nu\end{pmatrix}V^T,
\end{align*}
where $\tau_1\geq\tau_2\geq \nu>0$ are the singular values of $D_a\Phi_t(a_\ast)$. Let us define $\displaystyle \Lambda:=|\eta_3|$, as well as the matrix angular parameters
\begin{align*}
    \lambda_1&:=\frac{\tau_1|\eta_1|}{\nu\Lambda}\\
    \lambda_2&:=\frac{\tau_2|\eta_2|}{\nu\Lambda}\\
    \lambda:&=\sqrt{\lambda_1^2+\lambda_2^2}.
\end{align*}

We now introduce the dyadic region $E_{\lambda_1,\lambda_2,\Lambda}(a_*,t)$ defined by
\begin{align*}
    E_{\lambda_1,\lambda_2,\Lambda}(a_*,t):=\left\{|\eta_3|\simeq\Lambda,|\eta_1|\simeq\frac{\lambda_1\nu\Lambda}{\tau_1},|\eta_2|\simeq\frac{\lambda_2\nu\Lambda}{\tau_2}\right\},
\end{align*}
where $\lambda_1,\lambda_2,\Lambda\in 2^{\mathbb{Z}}$.

If $\eta_3$ is a given longitudinal coordinate, we also define its projection over the two coordinates $\eta_1$ and $\eta_2$ as
\begin{align*}
    E^\perp_{\lambda_1,\lambda_2,\Lambda}(a_*,t,\eta_3):=\left\{(\eta_1,\eta_2)\in\mathbb{R}^2|\,(\eta_1,\eta_2,\eta_3)\in E_{\lambda_1,\lambda_2,\Lambda}(a_*,t) \right\}.
\end{align*}

Given that we are working with generic labels $a_*$, for which the vorticity $\omega_0(a_*)$ does not necessarily vanish, in order to control the principal value integral in the expression of our driver, we shall also impose the \textbf{averaged Dini coherence condition}
\begin{align*}
    \sup_{a_*}\sum_{\lambda_1,\lambda_2,\Lambda}
\int_{\substack{E_{\lambda_1,\lambda_2,\Lambda}(a_*,t)\\|\Phi_t(a)-\Phi_t(a_\ast)|\leq L}}\frac{
\left|
\bigl[F(a,t)-F(a_*,t)\bigr]\omega_0(a)
\right|
}{
\left|\Phi_t(a)-\Phi_t(a_*)\right|^3
}\,da
\le
B_F(t),
\end{align*}
where
\begin{align*}
  \int_0^TB_F(t)\,dt<\infty.
\end{align*}
This condition resolves the principal value obstruction in the Biot-Savart Law for the strain matrix, relating the velocity to vorticity, since for general labels $a_*$, we cannot assume that $\omega_0(a_*)=0$, unlike for on-axis ones.

We shall also impose the following \textbf{near-field label condition}: There exists $L>0$ such that whenever $|\Phi_t(a)-\Phi_t(a_*)|\leq L$, there exist positive constants $c_1$ and $c_2$ independent of $t\in[0,T)$, $a$, and $a_*$ such that
\begin{align*}
    c_1|D_a\Phi_t(a_*)(a-a_*)|\leq|\Phi_t(a)-\Phi_t(a_*)|\leq c_2|D_a\Phi_t(a_*)(a-a_*)|.  
\end{align*}
This condition will ensure that the singular-value cone geometry describes the flow map.

Let $0<l_{clock}$ be a positive parameter, and we define

\begin{align*}
    E^{clock}_{\lambda_1,\lambda_2,\Lambda}(a_*,t):=E_{\lambda_1,\lambda_2,\Lambda}(a_*,t)\cap \{|a-a_\ast|<l_{clock}\}\cap\{|\Phi_t(a)-\Phi_t(a_\ast)|\leq L\}.
\end{align*}
and for a given $\eta_3$,
\begin{align*}
    E^{clock,\perp}_{\lambda_1,\lambda_2,\Lambda}(a_*,t,\eta_3):=\{(\eta_1,\eta_2)\in\mathbb{R}^2|(\eta_1,\eta_2,\eta_3)\in E^{clock}_{\lambda_1,\lambda_2,\Lambda}(a_*,t)\}.
\end{align*}
We shall refer to the union of these sets as the \textbf{clock-local region}.

Given a transversal coordinate $s$ with $|s|\simeq \Lambda$, we also define the \textbf{averaged transversal vorticity} $\overline{\omega}_{0,E}(s)$ as
\begin{align*}
  \overline{\omega}_{0,E}(s)&=\frac{1}{|E^{clock,\perp}_{\lambda_1,\lambda_2,\Lambda}(a_*,t,s)|}\int_{E^{clock,\perp}_{\lambda_1,\lambda_2,\Lambda}(a_*,t,s)}\omega_0\left(a_\ast+V(a_\ast,t)\begin{pmatrix}\eta_1\\\eta_2\\s\end{pmatrix}\right)\,d\eta_1d\,\eta_2. 
\end{align*}
When $|E^{clock,\perp}_{\lambda_1,\lambda_2,\Lambda}(a_\ast,t,s)|=0$, we shall make the convention $\overline{\omega}_{0,E}(s)=0$.

Let $a:=a_\ast+V\eta$, and in what follows, we shall make extensive use of the decomposition
\begin{align*}
    \omega_0(a)-\omega_0(a_\ast)=(\omega_0(a)-\overline{\omega}_{0,E}(\eta_3))+(\overline{\omega}_{0,E}(\eta_3)-\omega_0(a_\ast)).
\end{align*}
We shall regard the first term as a \textbf{transversal oscillation}, and the second one as a \textbf{longitudinal variation}.

We now impose the following \textbf{integrated transverse cusp coherence} condition
\begin{align*}
    \int_{E^{clock}_{\lambda_1,\lambda_2,\Lambda}}|\omega_0(a)-\overline{\omega}_{0,E}(\eta_3)|\,da\lesssim M_\alpha b_\Lambda(a_*,t)\left(\frac{\lambda\nu\Lambda}{\tau_2}\right)^\alpha\frac{\lambda_1\lambda_2\nu^2\Lambda^3}{\tau_1\tau_2}
\end{align*}
 where $b_\Lambda(a_*,t)$ are constants so that
\begin{align*}
\sup_{\substack{
a_*}}\sum_{\Lambda\lesssim l_{clock}}b_\Lambda(a_*,t)\Lambda^\alpha\leq \mathcal{T}_\alpha<\infty.
\end{align*}

We note that when $\Lambda\gg l_{clock}$, $E^{clock}_{\lambda_1,\lambda_2,\Lambda}$ is empty. Indeed, since $V(a_\ast,t)$ is orthogonal, it immediately follows that $\displaystyle |a-a_\ast|=|\eta|\geq|\eta_3|\simeq\Lambda\gg l_{clock}$.

 These estimates provide us with a reservoir-type bound, and within our Lagrangian setting, they play a role which is analogous to the one of the Lorentz norm $L^{3,1}$ present in the work of Danchin \cite{11}.

In order to control the contribution of the longitudinal variation, we also impose the following \textbf{longitudinal coherence} condition:
\begin{align*}
    \sup_{a_*}\sum_{\lambda_1,\lambda_2,\Lambda}\int_{E^{clock}_{\lambda_1,\lambda_2,\Lambda}}\frac{\tau_1|\overline{\omega}_{0,E}(\eta_3)-\omega_0(a_\ast)|}{|\Phi_t(a)-\Phi_t(a_\ast)|^3}\,da\leq B_{\parallel}(t),
\end{align*}
where $B_{\parallel}\in L_t^1([0,T))$.

In order to also control the contribution of the near field that is not contained in the clock-local region, we impose the following \textbf{off-clock coherence} condition
\begin{align*}
    \sup_{a_\ast}\int_{\substack{|\Phi_t(a)-\Phi_t(a_\ast)|\leq L\\|a-a_\ast|\geq l_{clock}}}\frac{\tau_1|\omega_0(a)-\omega_0(a_\ast)|}{|\Phi_t(a)-\Phi_t(a_\ast)|^3}\,da\leq B_R(t),
\end{align*}
where $B_{R}\in L_t^1([0,T))$.

Lastly, let us define $\kappa(t):=\frac{\tau_1(t)}{\tau_2(t)}$, and we shall impose the \textbf{bounded transverse distortion condition}, in the sense that
\begin{align*}
  \sup_{a_*}\sup_{t\in[0,T)}\kappa(t):=K_T<\infty. 
\end{align*}

This condition prevents an aspect-ratio loss.

We shall denote by $\mathcal{S}^\alpha_{\rm general}([0,T))$ the set of solutions satisfying the aforementioned conditions. If, in addition, the tail constant $\mathcal{T}_\alpha$, $K_T$, the near-field constants $c_1$, $c_2$, $L$, $l_{clock}$, $\|B_{\parallel}(t)\|_{L^1([0,T))}$, $\|B_R(t)\|_{L^1([0,T))}$, and $\|B_F(t)\|_{L^1([0,T))}$ are uniformly bounded in $t\in[0,T)$ and independent of the label $a_*$, we shall also say that the $\mathcal{S}^\alpha_{\rm general}([0,T))$ constants remain finite. All of these conditions should be interpreted as coherence hypotheses for a putative matrix-clock collapse.

\subsection{Axis labels}
We shall now describe the class of coherent conditional solutions $\mathcal{S}^\alpha_{\rm vert}$. Just like in the case of $\mathcal{S}^\alpha_{\rm general}$, we merely assume that these properties hold as part of the definition of our class of admissible solutions, and we do not pursue the proof of their propagation in this paper.

Let $a_*$ be an on-axis label. For an arbitrary label $a$, we write $\displaystyle a=a_*+\eta$, where $\eta=(\eta_1,
\eta_2,\eta_3)$. Let us once again define $\eta_\perp=(\eta_1,\eta_2)$. 

In this case, the symmetry assumptions readily imply that $\omega_0(a_*)=0$, and for this reason, we note that we shall no longer impose the Dini coherence condition.

As we shall also see in the proof of Theorem \ref{Main Result 2}, the two largest singular values $\tau_1$ and $\tau_2$ of $D_a\Phi_t(a_*)$ will be equal, so the bounded transverse distortion condition will automatically hold for on-axis labels. Another important aspect is that the deformation matrix $D_a\Phi_t(a_\ast)$ will turn out to be diagonal with entries $J_*^{-1}$, $J_*^{-1}$, and $J_*^2$.

We impose the following \textbf{near-field axis label condition}: there exist $L>0$ and $l_{clock}>0$ such that whenever $|\Phi_t(a)-\Phi_t(a_*)|\leq L$ and $|a-a_\ast|<l_{clock}$, 
 implies that
\begin{align*}
|\Phi_t(a)-\Phi_t(a_\ast)|&\simeq |D_a\Phi_t(a_\ast)(a-a_\ast)|\\
    r(\Phi_t(a))&\simeq\frac{R}{J_*},
\end{align*} 
along with $\displaystyle \|D\Phi_t(a)\|_{op}\simeq \|D\Phi_t(a_\ast)\|_{op}$, with uniform implicit constants.

In particular, this is also compatible with the essentially diagonal form of the deformation matrix $D_a\Phi_t(a_\ast)$ that we have previously claimed.

Let \begin{align*}
    \Lambda&=|\Delta Z|,\\
    \lambda&=\frac{R}{J_*^3\Lambda},
\end{align*}
and $l_{clock}>0$ a positive parameter.

We now introduce the dyadic block $E_{\lambda,\Lambda}$ defined by
\begin{align*}
    E_{\lambda,\Lambda}(a_\ast,t):=\left\{|\eta_3|\simeq\Lambda,|\eta_\perp|\simeq\lambda J^3_*\Lambda\right\},
\end{align*}
where $\lambda,\Lambda\in 2^{\mathbb{Z}}$, as well as the \textbf{clock-local} portion
\begin{align*}
    E^{clock}_{\lambda,\Lambda}(a_\ast,t):=E_{\lambda,\Lambda}(a_\ast,t)\cap\{|a-a_\ast|\leq l_{clock}\}\cap\{|\Phi_t(a)-\Phi_t(a_\ast)|\leq L\}.
\end{align*}

We now impose the following \textbf{integrated cusp/tail coherence condition} 
\begin{align*}
    \int_{E^{clock}_{\lambda,\Lambda}}|\omega_0(a)|\,da&\lesssim M_\alpha b_\Lambda(a_*,t)\lambda^{\alpha+2}J_\ast^{3\alpha+6}\Lambda^{\alpha+3},
\end{align*}
where $b_\Lambda(a_*,t)$ is a tail envelope satisfying
\begin{align*}
  \sup_{a_*}\sum_{\Lambda\leq l_{clock}}\Lambda^\alpha b_\Lambda(a_*,t)< \mathcal{T}_\alpha<\infty. 
\end{align*}
We once again note that when $\Lambda\gg l_{clock}$, the set $E^{clock}_{\lambda,\Lambda}$ is empty.

 We shall also impose the following \textbf{off-clock coherence} condition
\begin{align*}
    \sup_{a_\ast}\int_{\substack{|\Phi_t(a)-\Phi_t(a_\ast)|\leq L\\|a-a_\ast|\geq l_{clock}}}\frac{r(\Phi_t(a))|z(\Phi_t(a))-z(\Phi_t(a_\ast))|}{|\Phi_t(a)-\Phi_t(a_\ast)|^5}|\omega(\Phi_t(a),t)|\,da\leq B^{axis}_R(t),
\end{align*}
where $B^{axis}_{R}\in L_t^1([0,T))$.

We shall denote this class of solutions by $\mathcal{S}^\alpha_{\rm vert}([0,T))$.

We shall also say that the $\mathcal{S}^\alpha_{\rm vert}([0,T))$ constants remain finite if the tail constants $\mathcal{T}_\alpha$, $L$, $l_{clock}$, and $\|B^{axis}_R\|_{L_t^1}$, as well as the implicit near-field constants are uniform with respect to $t\in[0,T)$ and to the on-axis label $a_*$.
\section{Examples of solutions belonging to the coherent classes $\mathcal{S}^\alpha_{\rm general}$ and $\mathcal{S}^\alpha_{\rm vert}$}\label{Section 3}
In this section we shall provide broad classes of solutions belonging to our coherent classes $\mathcal{S}^\alpha_{\rm general}$ and $\mathcal{S}^\alpha_{\rm vert}$.

\begin{proposition}\label{General Class Existence}
    Let $\alpha\in (0,1)$. Let $u_0\in C^{1,\alpha}(\mathbb{R}^3)\cap L^2(\mathbb{R}^3)$ be axisymmetric initial data with no swirl which is also divergence-free. Let $u$ be a local solution of \eqref{Incompressible Euler} with initial data $u_0$. Then, there exists a positive time $T_0>0$ such that $u\in \mathcal{S}^\alpha_{\rm general}([0,T_0])$.
\end{proposition}
\begin{remark}
    We note that any continuation criterion over an interval $[0,T_\ast)$ would require the control quantities $B_F(t)$, $B_\parallel(t)$, $B_R(t)$ to be integrable and the near-field comparison constants and the transverse parameters $\kappa(t)$ to be uniformly controlled.
\end{remark}
\begin{proof}

    Classical local well-posedness theory, as mentioned in Chae-Wolf \cite{Chae-Wolf} and the references therein (the theory goes back to Lichtenstein \cite{Lichtenstein} and G\" unther \cite{Gunther}, while the Sobolev theory was treated by Kato \cite{Kato}, and more information can be found in Elgindi \cite{13}) shows that there exists a $C^{1,\alpha}\cap L^2$ solution with no swirl on the time interval $[0,T_0]$, for some $T_0>0$. For an arbitrary label $a$ and time $t\in[0,T_0]$, let $F(a,t):=D_a\Phi_t(a)$ be the deformation matrix.

    Let $\epsilon>0$ be a small parameter. The $C^{1,\alpha}$ regularity of the solution guarantees that, by shrinking $T_0$ if necessary, we may assume that $\|F(a,t)-I_3\|_{op}\leq \epsilon$, which means that the singular values of the deformation matrix satisfy $1-\epsilon\leq\tau_1,\tau_2,\nu\leq 1+\epsilon$. This allows us to infer that
    \begin{align*}
        \kappa(t)=\frac{\tau_1(t)}{\tau_2(t)}\leq\frac{1+\epsilon}{1-\epsilon}.
    \end{align*}
    Thus, the bounded transverse distortion condition is satisfied.

    The Fundamental Theorem of Calculus also implies that
    \begin{align*}
        \Phi_t(a)-\Phi_t(a_\ast)=\int_0^1F(a_\ast+v(a-a_\ast))(a-a_\ast)\,dv.
    \end{align*}
    Our norm operator bound now immediately implies that
    \begin{align*}
        (1-\epsilon)|a-a_\ast|\leq |\Phi_t(a)-\Phi_t(a_\ast)|\leq (1+\epsilon)|a-a_\ast|,
    \end{align*}
    and that $\displaystyle |F(a_\ast,t)(a-a_\ast)|\simeq |a-a_\ast|$ with uniform implicit constants, which shows that the near-field label condition is also satisfied.

    Since for every label $a$, the regularity of $u$ ensures that $F(a_*,t)$ is continuous, it follows that there exists a measurable choice for the orthogonal frame $V(a_\ast,t)$. In what follows, we shall also note that for our specific example, all of the estimates will be uniform with respect to $V$, so any choice would work.

    We now check the integrated transverse cusp coherence condition. Let $a_\ast$ be a label and $E^{clock}_{\lambda_1,\lambda_2,\Lambda}(a_\ast,t)$ be a clock-local block. Let $a=a_\ast+V\eta$, where $\eta=(\eta_1,\eta_2,\eta_3)$. Since $\omega_0\in C^{0,\alpha}$, for every $(\xi_1,\xi_2)\in E^{clock,\perp}_{\lambda_1,\lambda_2,\Lambda}$ we have
    \begin{align*}
        \left|\omega_0\left(a_\ast+V(a_\ast,t)\begin{pmatrix}\xi_1\\\xi_2\\\eta_3\end{pmatrix}\right)-\omega_0\left(a_\ast+V(a_\ast,t)\begin{pmatrix}\eta_1\\\eta_2\\\eta_3\end{pmatrix}\right)\right|&\lesssim \left|V(a_\ast,t)\begin{pmatrix}\xi_1-\eta_1\\\xi_2-\eta_2\\ 0\end{pmatrix}\right|^\alpha\\\
        &\lesssim \text{diam}(E^{clock,\perp}_{\lambda_1,\lambda_2,\Lambda})^\alpha\lesssim \left(\frac{\nu\lambda\Lambda}{\tau_2}\right)^\alpha.
    \end{align*}
    Averaging over $\displaystyle E^{clock,\perp}_{\lambda_1,\lambda_2,\Lambda}$ immediately gives the estimate
    \begin{align*}
        |\omega_0(a)-\overline{\omega}_{0,E}(\eta_3)|\lesssim \left(\frac{\nu\lambda\Lambda}{\tau_2}\right)^\alpha. 
    \end{align*}
    Upon integrating over $E^{clock}_{\lambda_1,\lambda_2,\Lambda}$, whose volume is $\displaystyle\frac{\nu^2\lambda_1\lambda_2\Lambda^3}{\tau_1\tau_2}$, immediately gives
    \begin{align*}
        \int_{E^{clock}_{\lambda_1,\lambda_2,\Lambda}} |\omega_0(a)-\overline{\omega}_{0,E}(\eta_3)|\,da\lesssim \left(\frac{\nu\lambda\Lambda}{\tau_2}\right)^\alpha\frac{\nu^2\lambda_1\lambda_2\Lambda^3}{\tau_1\tau_2}. 
    \end{align*}
    as claimed. Here, we note that we may take $b_\Lambda(a_\ast,t)\equiv C$, in which case we have
    \begin{align*}
        \sup_{a_\ast}\sum_{\substack{\Lambda\leq l_{clock}}}C\Lambda^\alpha<\infty,
    \end{align*}
    as desired.

    We now check the longitudinal coherence condition.

    For every $(\xi_1,\xi_2)\in E^{clock,\perp}_{\lambda_1,\lambda_2,\Lambda}$, we have
    \begin{align*}
        \left|\omega_0\left(a_\ast+V(a_\ast,t)\begin{pmatrix}\xi_1\\\xi_2\\\eta_3\end{pmatrix}\right)-\omega_0\left(a_\ast\right)\right|&\lesssim \left|\omega_0\left(a_\ast+V(a_\ast,t)\begin{pmatrix}\xi_1\\\xi_2\\\eta_3\end{pmatrix}\right)-\omega_0\left(a_\ast+V(a_\ast,t)\begin{pmatrix}\eta_1\\\eta_2\\\eta_3\end{pmatrix}\right)\right|\\
        &+\left|\omega_0\left(a_\ast+V(a_\ast,t)\begin{pmatrix}\eta_1\\\eta_2\\\eta_3\end{pmatrix}\right)-\omega_0\left(a_\ast\right)\right|\\
        &\lesssim \left|V(a_\ast,t)\begin{pmatrix}\xi_1-\eta_1\\\xi_2-\eta_2\\ 0\end{pmatrix}\right|^\alpha+\left|V(a_\ast,t)\begin{pmatrix}\eta_1\\\eta_2\\ \eta_3\end{pmatrix}\right|^\alpha\\\
        &\lesssim \text{diam}(E^{clock,\perp}_{\lambda_1,\lambda_2,\Lambda})^\alpha+|a-a_\ast|^\alpha\\
        &\lesssim \left(\frac{\nu\lambda\Lambda}{\tau_2}\right)^\alpha+|a-a_\ast|^\alpha.
    \end{align*}
    On the dyadic block $E_{\lambda_1,\lambda_2,\Lambda}$, we can immediately see that $|a-a_\ast|\simeq\nu\Lambda(1+\lambda^2)^{\frac{1}{2}}$. Since we have seen that $1-\epsilon\leq\tau_2\leq 1+\epsilon$, we immediately obtain the estimate
    \begin{align*}
\left|\omega_0\left(a_\ast+V(a_\ast,t)\begin{pmatrix}\xi_1\\\xi_2\\\eta_3\end{pmatrix}\right)-\omega_0\left(a_\ast\right)\right|&\lesssim |a-a_\ast|^\alpha.  
    \end{align*}
    Since we also know that $|\Phi_t(a)-\Phi_t(a_\ast)|\simeq |a-a_\ast|$, we have
    \begin{align*}
        \sum_{\substack{\lambda_1,\lambda_2,\Lambda}}\int_{E^{clock}_{\lambda_1,\lambda_2,\Lambda}}\frac{\tau_1|\overline{\omega}_{0,E}(\eta_3)-\omega_0(a_\ast)|}{|\Phi_t(a)-\Phi_t(a_\ast)|^3}\,da\lesssim \int_{|a-a_\ast|\leq L}\frac{1}{|a-a_\ast|^{3-\alpha}}\,da\leq C<\infty, 
    \end{align*}
    where $C$ is uniform in $a_\ast$. Here we have also used the fact that the dyadic clock-blocks $E^{clock}_{\lambda_1,\lambda_2,\Lambda}$ have a uniformly bounded number of overlaps.
    
    Moreover, since we have also seen that $\displaystyle |\omega_0(a)-\omega_0(a_\ast)|\lesssim |a-a_\ast|^\alpha$, we also have
    \begin{align*}
       \int_{\substack{|\Phi_t(a)-\Phi_t(a_\ast)|\leq L\\ |a-a_\ast|\geq l_{clock}}}\frac{\tau_1|\omega_0(a)-\omega_0(a_\ast)|}{|\Phi_t(a)-\Phi_t(a_\ast)|^3}\,da\lesssim \int_{|a-a_\ast|\leq L}\frac{1}{|a-a_\ast|^{3-\alpha}}\,da\leq C<\infty, 
    \end{align*}
     where $C$ is uniform in $a_\ast$. This proves both the longitudinal coherence and the off-clock condition, and shows that the control functions $B_{\parallel}$ and $B_{R}$ can be taken to be constants.

     We finally check the averaged Dini coherence condition.

     The regularity of $u$ guarantees that $F(t,a)=D_a\Phi_t(a)$ has $C^{0,\alpha}(\mathbb{R}^3)$ regularity in $a$, hence $\|F(a,t)-F(a_\ast,t)\|_{op}\lesssim |a-a_\ast|^\alpha$.

     We immediately get that
\begin{align*}
      \sum_{\substack{\lambda_1,\lambda_2,\Lambda}} \int_{\substack{E_{\lambda_1,\lambda_2,\Lambda}\\|\Phi_t(a)-\Phi_t(a_\ast)|\leq L}}&\frac{\|F(a,t)-F(a_\ast,t)\|_{op}|\omega_0(a)|}{|\Phi_t(a)-\Phi_t(a_\ast)|^3}\,da\\
      &\lesssim\int_{|\Phi_t(a)-\Phi_t(a_\ast)|\leq L}\frac{\|F(a,t)-F(a_\ast,t)\|_{op}|\omega_0(a)|}{|\Phi_t(a)-\Phi_t(a_\ast)|^3}\,da\\ &\lesssim\int_{|a-a_\ast|\leq L}\frac{1}{|a-a_\ast|^{3-\alpha}}\,da\leq C<\infty, 
    \end{align*}
     where $C$ is uniform in $a_\ast$. This shows that $B_F$ can also be chosen to be constant.

     This finishes the proof.
\end{proof}
\begin{proposition}\label{Axis Class Existence}
 Let $\alpha\in (0,1)$. Let $u_0\in C^{1,\alpha}(\mathbb{R}^3)\cap L^2(\mathbb{R}^3)$ be axisymmetric initial data with no swirl which is also divergence-free and odd in $z$. Let $u$ be a local solution of \eqref{Incompressible Euler} with initial data $u_0$. Then, there exists a positive time $T_0>0$ such that $u\in \mathcal{S}^\alpha_{\rm vert}([0,T_0])$.    
\end{proposition}
\begin{remark}
    We note that any continuation criterion over an interval $[0,T_\ast)$ would require the control quantity $B^{axis}_R(t)$ to be integrable and the near-field axis comparison and integrated transverse cusp coherence constants to be uniformly controlled.
\end{remark}
\begin{proof}
 Just as in the proof of Proposition \ref{General Class Existence}, if the initial data $u_0$ belongs to $C^{1,\alpha}(\mathbb{R}^3)\cap L^2(\mathbb{R}^3)$, has no swirl, is axisymmetric, odd in $z$, and has divergence zero, the classical local well-posedness theory mentioned in Chae-Wolf \cite{Chae-Wolf} and the references therein (the theory goes back to Lichtenstein \cite{Lichtenstein} and G\" unther \cite{Gunther}, while the Sobolev theory was treated by Kato \cite{Kato}, and more information can be found in Elgindi \cite{13}) ensures the existence of a $C^{1,\alpha}(\mathbb{R}^3)\cap L^2(\mathbb{R}^3)$ axisymmetric solution $u$ with no swirl on the time interval $[0,T_0]$, for some $T_0>0$.  It is also clear that the vorticity
    \begin{align*}
        \omega_0=\omega_0^\theta e^\theta
    \end{align*}
    has $C^{0,\alpha}(\mathbb{R}^3)$. We also note that the symmetry conditions guarantee that $\omega_0(a)=0$ for every point on the $z$-axis.

    Let $\epsilon>0$ be a small parameter. For an arbitrary label $a$, we set $F(a,t)=D_a\Phi_t(a)$. The regularity of the solution guarantees that, by shrinking $T_0$ if necessary, we may assume that $\|F(a,t)-I_3\|_{op}\leq \epsilon$, which means that the singular values of the deformation matrix $F(a,t)$ satisfy the inequality $1-\epsilon\leq\tau_1(t),\tau_2(t),\nu(t)\leq 1+\epsilon$ for every $t\in[0,T_0]$. Repeating the argument from the proof of Proposition \ref{General Class Existence}, we immediately get that
    \begin{align*}
    |\Phi_t(a)-\Phi_t(a_\ast)|&\simeq |D_a\Phi_t(a_\ast)(a-a_\ast)|\\
        \|D_a\Phi_t(a_\ast)\|_{op}&\simeq \|D_a\Phi_t(a)\|_{op},
    \end{align*}
    with uniform implicit constants. Now, given that we are in the axisymmetric with no swirl setting, we may write the flow as
\begin{align*}
 \Phi_t(a_*)=(\varphi^r(R,Z,t)\cos\Theta,\varphi^r(R,Z,t)\sin\Theta,\varphi^z(R,Z,t)),   \end{align*}
 hence as in \cite{Shkoller}, the deformation matrix can be seen to be given by
 \begin{align*}
  D_a\Phi_t(a_\ast)&=\begin{pmatrix}
\partial_R\varphi^r & 0 & \partial_Z\varphi^r\\
0 & \frac{\varphi^r}{R} & 0\\
\partial_R\varphi^z & 0 & \partial_Z\varphi^z
      \end{pmatrix}   
 \end{align*}
    Since $J_\ast\simeq 1$, this also immediately implies that $\displaystyle r(\Phi_t(a))\simeq R$, with uniform implicit constants, and verifies the near-field axis label condition.

    We now check the integrated cusp/tail coherence condition. Let $a\in E_{\lambda,\Lambda}(a_\ast,t)$. If we set $a=a_\ast+\eta$, since $\omega_0\in C^{0,\alpha}$, we have
    \begin{align*}
        |\omega_0(a)|=|\omega_0(a)-\omega(0,0,a_3)|\lesssim |(a_1,a_2,0)|^\alpha=|\eta_\perp|^\alpha\lesssim \lambda^\alpha J^{3\alpha}\Lambda^\alpha.
    \end{align*}
    Since the volume of $E_{\lambda,\Lambda}(a_\ast,t)$ is $\displaystyle \lambda^2 J_\ast^6\Lambda^3$, integrating over $E^{clock}_{\lambda,\Lambda}(a_\ast,t)\subset E_{\lambda,\Lambda}(a_\ast,t)$ gives the estimate 
    \begin{align*}
        \int_{E^{clock}_{\lambda,\Lambda}(a_\ast,t)} |\omega_0(a)|\,da\leq \int_{E_{\lambda,\Lambda}(a_\ast,t)} |\omega_0(a)|\,da&\lesssim \lambda^{\alpha+2} J^{3\alpha+6}\Lambda^{\alpha+3}.
    \end{align*}
    We next set $b_\Lambda(a_\ast,t)\equiv C$, and we see that
    \begin{align*}
\sup_{a_\ast}\sum_{\substack{\Lambda\leq l_{clock}}}\Lambda^\alpha<\infty.
    \end{align*}
    This proves the integrated cusp/tail coherence property.

    Lastly, we move on to check off-clock condition. By the Cauchy property, 
    \begin{align*}
\omega(t,\Phi_t(a_\ast))=D_a\Phi_t(a_\ast)\omega_0(a_*)=0,
    \end{align*}
    hence
    \begin{align*}
        r(\Phi_t(a))&\leq |\Phi_t(a)-\Phi_t(a_\ast)|\\
        |z(\Phi_t(a))-z(\Phi_t(a_\ast))|&\leq |\Phi_t(a)-\Phi_t(a_\ast)|.
    \end{align*}
    Together with the H\"older estimate $\displaystyle |\omega_0(a)-\omega_0(a_\ast)|\lesssim |a-a_\ast|^\alpha$,  the near-field axis condition immediately implies that
    \begin{align*}
        \int_{\substack{|\Phi_t(a)-\Phi_t(a_\ast)|\leq L\\|a-a_\ast|\geq l_{clock}}}&\frac{r(\Phi_t(a))|z(\Phi_t(a))-z(\Phi_t(a_\ast))|}{|\Phi_t(a)-\Phi_t(a_\ast)|^5}|\omega(\Phi_t(a),t)|\,da\\
        &\lesssim \int_{\substack{|\Phi_t(a)-\Phi_t(a_\ast)|\leq L\\|a-a_\ast|\geq l_{clock}}}\frac{|\Phi_t(a)-\Phi_t(a_\ast)|^2}{|\Phi_t(a)-\Phi_t(a_\ast)|^5}|a-a_\ast|^\alpha\,da\\
        &\lesssim \int_{\substack{L\gtrsim|a-a_\ast|\geq l_{clock}}}\frac{1}{|a-a_\ast|^{3-\alpha}}\,da\leq C<\infty,
    \end{align*}
    where $C$ is uniform in $a_*$. This shows that the function $B^{axis}_R$ can be chosen to be constant.
\end{proof}
\begin{remark}
    Propositions \ref{General Class Existence} and \ref{Axis Class Existence} in particular imply that our coherent classes of solutions that we have defined are not contained in the classes of solutions obtained by Danchin in \cite{11} and Shao-Wei-Zhang in \cite{Shao-Wei-Zhang}. In order to see that, we note that it suffices to choose initial data $u_0\in C^{1,\alpha}(\mathbb{R}^3)\cap L^2(\mathbb{R}^3)$ (with the extra odd symmetry in $z$ if the coherence class $\mathcal{S}^\alpha_{\rm vert}$ is concerned) such that the specific vorticity $\displaystyle \xi_0=\frac{\omega^\theta}{r}\notin L^{3,\infty}(\mathbb{R}^3)$.
    
    In this case, the resulting classical solutions, which belong to our coherent classes by the aforementioned propositions, lie outside of the class considered by Shao-Wei-Zhang, and since $L^{3,1}(\mathbb{R}^3)\subset L^{3,\infty}(\mathbb{R}^3)$, it also lies outside of the class considered by Danchin. 
    
    Therefore, the properties defining our classes of coherent solutions are not reformulations of the global Lorentz assumptions present in \cite{11} and \cite{Shao-Wei-Zhang}. On the other hand, we must also note that unlike in the aforementioned results, our solutions are guaranteed to not blow up in finite time only for as long as the coherence properties defining $\mathcal{S}^\alpha_{\rm general}$ and $\mathcal{S}^\alpha_{\rm vert}$ hold.
\end{remark}
\section{The proof of Theorem \ref{Main Result}}\label{Section 4}
 Let $a_*=(R,\Theta,Z)$ be a generic label and $x_*(t)=\Phi_t(a_*)$ be the corresponding (potentially collapsing) trajectory.
Let
\begin{align*}
    F_*(t):&=D_a\Phi_t(a_*)
\end{align*}
be the deformation matrix.
Given that we are in the axisymmetric with no swirl setting, we may write the flow as
\begin{align*}
 \Phi_t(a_*)=(\varphi^r(R,Z,t)\cos\Theta,\varphi^r(R,Z,t)\sin\Theta,\varphi^z(R,Z,t)),   \end{align*}
 hence as in \cite{Shkoller}, the deformation matrix can be seen to be given by
 \begin{align*}
  F_*(t):&=\begin{pmatrix}
\partial_R\varphi^r & 0 & \partial_Z\varphi^r\\
0 & \frac{\varphi^r}{R} & 0\\
\partial_R\varphi^z & 0 & \partial_Z\varphi^z
      \end{pmatrix}   
 \end{align*}
  The middle entry $\displaystyle \lambda_\theta=\frac{\varphi^r}{R}$ is called the \textbf{azimuthal/hoop stretch}, in the sense that it essentially maps a circle of radius $R$ to one of radius $\varphi^r$.
 
We consider the singular-value decomposition of $F_*(t)$, which is given by
\begin{align*}
    F_*(t)&=U\begin{pmatrix}\tau_1 & 0 &0\\0 & \tau_2 & 0\\
    0 & 0& \nu\end{pmatrix}V^T,
\end{align*}
where $\tau_1\geq\tau_2\geq\nu>0$ are the singular values of $F_*$.

From the incompressibility condition, we also know that $\tau_1\tau_2\nu=1$. 

For our potentially collapsing matrix clock, we shall choose $\nu(t)=\sigma_{min}(D_a\Phi_t(a_*))$. Let $E_{\nu(t)}$ be the  left singular subspace associated to $\nu(t)$. 

We have the following
\begin{lemma}\label{Driver ODE}
In the previous notations, we have the relation
\begin{align*}
    \frac{d^+}{dt}\nu(t)|_{t=t_0}&=\nu(t_0)\inf_{\substack{e\in E_{\nu(t_0)},|e|=1}}e\cdot S(x_*(t_0),t_0)e,
\end{align*}
where $\displaystyle S=\frac{1}{2}(\nabla u+\nabla u^T)$ is the strain matrix and $d^+$ is the right Dini derivative.
\end{lemma}
\begin{proof}
  Let $F(t):=D_a\Phi_t(a)$ along the trajectory $x(t)=\Phi_t(a)$. We have $\dot{F}(t)=\nabla u(x(t),t)F(t)$.

  We consider the left singular subspace $E_{\nu(t)}$ associated to $\nu(t)$, which can be viewed as the eigenspace of $C(t):=F(t)F(t)^T$ corresponding to the eigenvalue $(\nu(t))^2$.

It is also clear that $\lambda(t):=(\nu(t))^2$ is the smallest eigenvalue of the positive definite matrix $C(t)$. Let $L(t):=\nabla u(x(t),t)$

The time derivative of $C(t)$ is given by the relation
\begin{align*}
 \dot{C(t)}&= LFF^T+FF^TL^T=LC+CL^T
\end{align*}

  We shall also need the following standard formula for the smallest eigenvalue of the symmetric matrix $C(t)$:
  \begin{align*}
      \frac{d^+}{dt}\lambda(t)_{|t=t_0}=\min_{\substack{e\in E_{\nu(t_0)}\\|e|=1}}e\cdot \dot{C}(t_0)e.
  \end{align*}
  We note that the minimum is well-defined, since it is taken over a compact set.
  
  For the sake of completeness, we shall first prove this formula.

  We know that $C(t)$ is differentiable. Then,
  \begin{align*}
      C(t_0+h)&=C(t_0)+h\dot{C}(t_0)+o(h)
  \end{align*}
  in the sense of operator norms as $h\rightarrow 0$.

  Let $\displaystyle \tau:=\min_{\substack{e\in E_{\nu(t_0)}\\|e|=1}}e\cdot \dot{C}(t_0)e$.

  We first prove the lower bound
  \begin{align*}
      \liminf_{\substack{h\rightarrow 0^+}}\frac{\lambda(t_0+h)-\lambda(t_0)}{h}\geq \tau.
  \end{align*}

  Let $v$ be an arbitrary unit vector, which we shall decompose as $v=v_E+v_\perp$, where $v_E\in E_{\nu(t_0)}$, and $v_\perp\perp E_{\nu(t_0)}$.

  We first analyze the case $E_{\nu(t_0)}\neq\mathbb{R}^3$, and let $\gamma>\lambda(t_0)$ be the next eigenvalue of $C(t_0)$. In this case, it is clear that
  \begin{align*}
      v\cdot C(t_0)v\geq \lambda(t_0)+(\gamma-\lambda(t_0))|v_\perp|^2
  \end{align*}

  We also have the inequality
  \begin{align*}
      v\cdot \dot{C}(t_0)v=v_E\cdot \dot{C}(t_0)v_E+2v_\perp\cdot \dot{C}(t_0)v_E+v_\perp\cdot \dot{C}(t_0)v_\perp,
  \end{align*}
  Let $x:=|v_E|$, with $0\leq x\leq 1$, so that $|v_\perp|=\sqrt{1-x^2}$.
It then follows that
\begin{align*}
      v\cdot \dot{C}(t_0)v&=v_E\cdot \dot{C}(t_0)v_E+2v_\perp\cdot \dot{C}(t_0)v_E+v_\perp\cdot \dot{C}(t_0)v_\perp\\
      &\geq x^2\tau -2\|\dot{C}(t_0)\|_{op}x\sqrt{1-x^2}-\|\dot{C}(t_0)\|_{op}(1-x^2)\\
      &=\tau-2\|\dot{C}(t_0)\|_{op}x\sqrt{1-x^2}-(\tau+\|\dot{C}(t_0)\|_{op})(1-x^2)\\
      &\geq \tau-4\|\dot{C}(t_0)\|_{op}\sqrt{1-x^2}=\tau-4\|\dot{C}(t_0)\|_{op}|v_\perp|
  \end{align*}
It follows that
\begin{align*}
    v\cdot C(t_0+h)v&\geq \lambda(t_0)+(\gamma-\lambda)|v_\perp|^2+h\tau-4\|\dot{C}(t_0)\|_{op}h|v_\perp|+o(h)\\
    &\geq\lambda(t_0)+h\tau-\left(\frac{2\|\dot{C}(t_0)\|_{op}}{\sqrt{\gamma-\lambda(t_0)}}\right)^2h^2+o(h)\\
    &=\lambda(t_0)+h\tau+o(h)
\end{align*}
By taking the infimum over $v$ with $|v|=1$, we get that
\begin{align*}
    \lambda(t_0+h)\geq\lambda(t_0)+h\tau+o(h),
\end{align*}
which immediately shows that
\begin{align*}
      \liminf_{\substack{h\rightarrow 0^+}}\frac{\lambda(t_0+h)-\lambda(t_0)}{h}\geq \tau.
  \end{align*}

  We now analyze the case $E_{\nu(t_0)}=\mathbb{R}^3$. This time, $C(t_0)=\lambda(t_0)Id$, hence
  \begin{align*}
      v\cdot C(t_0)v=\lambda(t_0).
  \end{align*}

  By definition, it is also clear that

  \begin{align*}
      v\cdot \dot{C}(t_0)v\geq\tau,
  \end{align*}
  hence
  \begin{align*}
      v\cdot C(t_0+h)v\geq \lambda(t_0)+h\tau+o(h).
  \end{align*}
  By taking the infimum over $v$ with $|v|=1$, we get that
\begin{align*}
    \lambda(t_0+h)\geq\lambda(t_0)+h\tau+o(h),
\end{align*}
which is the same inequality from above, and once again leads directly to
\begin{align*}
      \liminf_{\substack{h\rightarrow 0^+}}\frac{\lambda(t_0+h)-\lambda(t_0)}{h}\geq \tau.
  \end{align*}
  We now prove the upper bound
  \begin{align*}
      \limsup_{\substack{h\rightarrow 0^+}}\frac{\lambda(t_0+h)-\lambda(t_0)}{h}\leq \tau.
  \end{align*}
Let $e_0\in E_{\nu(t_0)}$ be a vector attaining the minimum. Then,
\begin{align*}
    \lambda(t_0+h)\leq e_0\cdot C(t_0+h)e_0=\lambda(t_0)+h\tau+o(h)
\end{align*}
The desired bound
\begin{align*}
      \limsup_{\substack{h\rightarrow 0^+}}\frac{\lambda(t_0+h)-\lambda(t_0)}{h}\leq \tau
  \end{align*}
  now immediately follows.

We now recall that $\lambda(t)=(\nu(t))^2$, which means that
\begin{align*}
    \frac{d^+}{dt}_{|t=t_0}\lambda(t)=2\nu(t_0)\frac{d^+}{dt}_{|t=t_0}\nu(t)
\end{align*}

For $e\in E_{\nu(t_0)}$, we also have that
\begin{align*}
  e\cdot \dot{C}(t_0)e=e\cdot LCe+e\cdot CL^Te=(\nu(t_0))^2 e\cdot Le+(\nu(t_0))^2 e\cdot L^Te=2(\nu(t_0))^2 e\cdot S(t_0)e  
\end{align*}
 We now deduce that
 \begin{align*}
    2\nu(t_0)\frac{d^+}{dt}_{|t=t_0}\nu(t)=2(\nu(t_0))^2 \min_{\substack{e\in E_{\nu(t_0)}\\|e|=1}}e\cdot S(t_0)e.
\end{align*}
Upon dividing by $\nu(t_0)>0$, we obtain the desired conclusion.
\end{proof}
Let us fix $e\in E_{\nu(t)}$. Let us now define the driver $\mathcal{D}_*(t):=-e\cdot S(x_*(t),t)e$.

We then have the following
\begin{proposition}\label{Driver Estimate}
Let $u\in\mathcal{S}_{\rm general}^\alpha([0,T_*))$ satisfy the hypotheses of Theorem \ref{Main Result}. If $T_*$ is finite, then there exist a positive function $B\in L^1([0,T_*))$ and a positive finite constant $C_{T_*}$, independent of $a_*$ and $e$ with $\|e\|=1$, such that 
\begin{align*}
    |\mathcal{D}_*(t)|&\leq B(t)+C_{T_*}\,\nu(t)^{\frac{3\alpha-1}{2}},
\end{align*}
for every $t\in[0,T_*)$.
\end{proposition}
\begin{remark}
    We shall actually only need an upper bound for the driver $\mathcal{D}_*(t)$, so the result that we prove is actually stronger than what is required for the proof of Theorem \ref{Main Result}.
\end{remark}
For the rest of this section, we shall assume that Proposition \ref{Driver Estimate} is true, and we shall show how it implies Theorem \ref{Main Result}.

Let $T_*>0$ be the maximal time of existence for a solution satisfying the hypotheses of Theorem \ref{Main Result}. We assume for the sake of contradiction that $T_*<\infty$. 

We recall that we have chosen $a_*$ to be an arbitrary label. By Lemma \ref{Driver ODE}, we have
\begin{align*}
    \frac{d^+}{dt}\nu(t)&=\nu(t)\inf_{\substack{e\in E_{\nu(t)},|e|=1}}e\cdot S(x_*(t),t)e\geq -B(t)\nu(t) -C_{T_*}\,\nu(t)^{\frac{3\alpha+1}{2}},
\end{align*}
where the constants are independent of $a_*$, by the hypotheses of Theorem \ref{Main Result}.

When $\alpha=\frac{1}{3}$, then the inequality becomes
\begin{align*}
    \frac{d^+}{dt}\nu(t)&\geq -(B(t)+C_{T_*})\nu(t).
\end{align*}
A direct application of Gr\"onwall's inequality, together with the fact that $\nu(0)=1$, lead to the estimate
\begin{align*}
    \nu(t)\geq e^{-\int_0^tB(\tau)\,d\tau-C_{T_*}t},
\end{align*}
for every $t\in[0,T_*)$.

When $\alpha>\frac{1}{3}$, $\frac{3\alpha+1}{2}>1$, hence the comparison ODE becomes
\begin{align*}
    \dot{f}(t)&=-B(t)f(t)-C_{T_*} f(t)^{\frac{3\alpha+1}{2}},
\end{align*}
with initial data $f(0)=1$.

By dividing by $f(t)^{\frac{3\alpha+1}{2}}$ and setting $\displaystyle g(t):=\frac{f(t)^{\frac{1-3\alpha}{2}}}{\frac{1-3\alpha}{2}}$, we obtain the ODE

\begin{align*}
    \dot{g}(t)&=-\frac{1-3\alpha}{2}B(t)g(t)-C_{T_*},
\end{align*}
and another direct application of Gr\"onwall's inequality allows us to obtain a uniform bound on $-g$ on $[0,T_\ast)$. Since the exponent of $f$ in the definition of $g$ is negative, this implies that $f$ cannot hit zero in finite time. Therefore, $\nu(t)$ cannot hit zero in finite time either.

Thus, there exists $c_{T_*}>0$ independent of the label $a_*$ such that $\nu(t)\geq c_{T_*}$, for every $t\in[0,T_*)$. Now, the relations $\tau_1\geq\tau_2\geq\nu$ and $\tau_1\tau_2\nu=1$ imply that $\tau_1\leq c_{T_*}^{-2} $, hence we necessarily have $\|F(a_*,t)\|_{op}\leq c_{T_*}^{-2}$ for every label $a_*$.

Since for every $a\in\mathbb{R}^3$, $\omega(\Phi_t(a))=F(a,t)\omega_0(a)$, it follows that 

\begin{align*}
    \|\omega(\Phi_t(a))\|_{L^\infty}\leq \sup_{a}\|F(a,t)\|_{\rm op}\|\omega_0\|_{L^\infty},
\end{align*}
hence 
\begin{align*}
  \int_0^{T_*}\|\omega(t)\|_{L^\infty}\,dt\leq T_*c_{T_*}^{-2}\|\omega_0\|_{L^\infty}<\infty. 
\end{align*}
By the Beale-Kato-Majda criterion, the solution $u$ continues past $T_*$, which contradicts maximality. Therefore, $T_*$ is not a blow-up time.

\section{The proof of Proposition \ref{Driver Estimate}}\label{Section 5}
In this section, we shall provide the proof of Proposition \ref{Driver Estimate}. The main idea will be to split the analysis into far and near labels.

By the Biot-Savart Law \eqref{Biot-Savart}, the driver will have the representation
\begin{align*}
    \mathcal{D}&=-e\cdot S(x_*)e=\text{p.v.}\int_{\mathbb{R}^3}K_e(x_*-y)\omega(y,t),
\end{align*}
where $K_e$ is a Calderón-Zygmund kernel with the properties
\begin{align*}
    |K_e(h)|&\lesssim \frac{1}{|h|^3}\\
    \int_{\mathbf{S}^2}K_e(\theta)\,d\theta&=0
\end{align*}
Let $L>0$ be given by our near-field label condition. We also define $\chi_L=\chi(\cdot/L)$, where $\chi$ is a smooth radial function which is supported in $\displaystyle\left\{ |z|\geq\frac{1}{2}\right\}$, and is equal to $1$ on $\displaystyle\{ |z|\geq 1\}$.

We may then write
\begin{align*}
    \mathcal{D}&=p.v\int_{\mathbb{R}^3}K_e(x_*-y)\omega(y,t)\,dy\\
    &=p.v\int_{\mathbb{R}^3}(1-\chi_L(x_*-y))K_e(x_*-y)\omega(y,t)\,dy+p.v\int_{\mathbb{R}^3}\chi_L(x_*-y)K_e(x_*-y)\omega(y,t)\,dy\\
     &=\sum_{k=-\infty}^0p.v.\int_{2^{k-1}L\leq|x_*-y|\leq 2^{k}L}(1-\chi_L(x_*-y))K_e(x_*-y)\omega(y,t)\,dy\\
     &+\int \chi_L(x_*-y)K_e(x_*-y)\omega(y,t)\,dy\\
    &=\sum_{k=-\infty}^0\int_{2^{k-1}L\leq|x_*-y|\leq 2^{k}L}(1-\chi_L(x_*-y))K_e(x_*-y)(\omega(y,t)-\omega(x_*,t))\,dy\\
    &+\int_{|x_*-y|\geq L}\chi_L(x_*-y)K_e(x_*-y)\omega(y,t)\,dy\\
    &=\int_{|x_*-y|\leq L}(1-\chi_L(x_*-y))K_e(x_*-y)(\omega(y,t)-\omega(x_*,t))\,dy\\
    &+\int_{|x_*-y|\geq \frac{L}{2}}\chi_L(x_*-y)K_e(x_*-y)\omega(y,t)\,dy\\
    &:=\mathcal{D}_{\rm near}+\mathcal{D}_{\rm far}.
\end{align*}
We shall first analyze the estimate for the near labels, when $|x_*-y|\leq L$.

We have to bound

\begin{align*}
    \int_{|x_*-y|\leq L}(1-\chi_L(x_*-y))K_e(x_*-y)(\omega(y,t)-\omega(x_*,t))\,dy
\end{align*}

For an arbitrary label $a$, we write $\displaystyle a=a_*+V(a_\ast,t)\eta$, where $\eta=(\eta_1,
\eta_2,\eta_3)$. We also recall the notation $\eta_\perp:=(\eta_1,\eta_2)$.

When $|\Phi_t(a)-\Phi_t(a_\ast)|\leq L$, the near-field label condition implies that 
\begin{align*}
\Phi_t(a)-\Phi_t(a_*)\simeq U\begin{pmatrix}\tau_1\eta_1\\\tau_2\eta_2\\\nu\eta_3\end{pmatrix}.
\end{align*}
It also follows that the Eulerian distance is given by
\begin{align*}
    d(a,t):&=|\Phi_t(a)-\Phi_t(a_*)|\simeq\sqrt{\tau_1^2\eta_1^2+\tau_2^2\eta_2^2+\nu^2\eta_3^2}.
\end{align*}

In particular, in the dyadic region $E_{\lambda_1,\lambda_2,\Lambda}$, the Eulerian distance will satisfy 
\begin{align*}
    d(a,t)&\simeq \nu\Lambda\sqrt{1+\lambda_1^2+\lambda_2^2}.
\end{align*}

We shall now split our analysis into two subregions: a cone region, in which $\lambda\simeq 1$, and an off-cone region, where $\lambda\not\simeq 1$. We shall first focus on the latter.

Our quantity of interest here is
\begin{align*}
\mathcal{D}_{\rm near}&=\int_{|x_*-y|\leq L}(1-\chi_L(x_*-y))K_e(x_*-y)(\omega(y,t)-\omega(x_*,t))\,dy
\end{align*}

Thanks to the cancellation in our integrand, we shall be able to repeatedly make use of the bound
\begin{align*}
    |K_e(x_*-y)|\lesssim \frac{1}{|x_*-y|^3}.
\end{align*}

Let $y=\Phi_t(a)$ for a suitable label $a$.

We also know that
\begin{align*}
    \omega(\Phi_t(a),t)&=F(a,t)\omega_0(a),
\end{align*}
hence
\begin{align*}
   \omega(\Phi_t(a),t)-\omega(\Phi_t(a_*),t)&=F(a,t)\omega_0(a)-F(a_*,t)\omega_0(a_*)\\
   &=(F(a,t)-F(a_*,t))\omega_0(a)+F(a_*,t)(\omega_0(a)-\omega_0(a_*)) 
\end{align*}

Let the integral contribution over the dyadic block $E_{\lambda_1,\lambda_2,\Lambda}$ be given by $\mathcal{D}_{\lambda_1,\lambda_2,\Lambda}$ (here, $\lambda_1,\lambda_2,\Lambda\in 2^{\mathbb{Z}}$). More precisely,
\begin{align*}
    \mathcal{D}_{\lambda_1,\lambda_2,\Lambda}&=\int_{\substack{|x*-y|\leq L\\
    y\in \Phi_t(E_{\lambda_1,\lambda_2,\Lambda})}}K_e(x_*-y)(\omega(y,t)-\omega(x_*,t))\,dy\\
    &=\int_{\substack{|\Phi_t(a_*)-\Phi_t(a)|\leq L\\
    a\in E_{\lambda_1,\lambda_2,\Lambda}}}K_e(\Phi_t(a_*)-\Phi_t(a))(\omega(\Phi_t(a),t)-\omega(\Phi_t(a_*),t))\,da\\
    &=\int_{\substack{|\Phi_t(a_*)-\Phi_t(a)|\leq L\\
    a\in E_{\lambda_1,\lambda_2,\Lambda}}}K_e(\Phi_t(a_*)-\Phi_t(a))F(a_*,t)(\omega_0(a)-\omega_0(a_*))\,da\\
    &+\int_{\substack{|\Phi_t(a_*)-\Phi_t(a)|\leq L\\
    a\in E_{\lambda_1,\lambda_2,\Lambda}}}K_e(\Phi_t(a_*)-\Phi_t(a))(F(a,t)-F(a_*,t))\omega_0(a)\,da:=I_{\lambda_1,\lambda_2,\Lambda}+II_{\lambda_1,\lambda_2,\Lambda}
\end{align*}

In the change of variable, we have used the volume-preserving property of the incompressible flow.
We further split
\begin{align*}
   I_{\lambda_1,\lambda_2,\Lambda}=\int_{\substack{|\Phi_t(a_*)-\Phi_t(a)|\leq L\\
    a\in E_{\lambda_1,\lambda_2,\Lambda}}}&K_e(\Phi_t(a_*)-\Phi_t(a))F(a_*,t)(\omega_0(a)-\omega_0(a_*))\,da\\
    &=\int_{\substack{|a-a_\ast|\leq l_{clock}\\|\Phi_t(a_*)-\Phi_t(a)|\leq L\\
    a\in E_{\lambda_1,\lambda_2,\Lambda}}}K_e(\Phi_t(a_*)-\Phi_t(a))F(a_*,t)(\omega_0(a)-\overline{\omega}_{0,E}(\eta_3))\,da\\
    &+\int_{\substack{|a-a_\ast|\leq l_{clock}\\|\Phi_t(a_*)-\Phi_t(a)|\leq L\\
    a\in E_{\lambda_1,\lambda_2,\Lambda}}}K_e(\Phi_t(a_*)-\Phi_t(a))F(a_*,t)(\overline{\omega}_{0,E}(\eta_3)-\omega_0(a_*))\,da\\
    &+\int_{\substack{|a-a_\ast|>l_{clock}\\|\Phi_t(a_*)-\Phi_t(a)|\leq L\\
    a\in E_{\lambda_1,\lambda_2,\Lambda}}}K_e(\Phi_t(a_*)-\Phi_t(a))F(a_*,t)(\omega_0(a)-\omega_0(a_*))\,da\\
    &:=I_{\lambda_1,\lambda_2,\Lambda,clock,\perp}+I_{\lambda_1,\lambda_2,\Lambda,clock,\parallel}+I_{\lambda_1,\lambda_2,\Lambda,off-clock}
\end{align*}

We move on to the longitudinal contribution. Since $\|F(a_\ast,t)\|_{op}\leq\tau_1$, we have

\begin{align*}
 |I_{\lambda_1,\lambda_2,\Lambda,clock,\parallel}|&\lesssim \int_{\substack{|a-a_\ast|\leq l_{clock}\\|\Phi_t(a_*)-\Phi_t(a)|\leq L\\
    a\in E_{\lambda_1,\lambda_2,\Lambda}}}\frac{\tau_1|\overline{\omega}_{0,E}(\eta_3)-\omega_0(a_*)|}{|\Phi_t(a_*)-\Phi_t(a)|^3}\,da   
\end{align*}
The longitudinal coherence condition now guarantees that
\begin{align*}
 \sum_{\substack{\lambda_1,\lambda_2,\Lambda}}|I_{\lambda_1,\lambda_2,\Lambda,clock,\parallel}|&\lesssim \sum_{\substack{\lambda_1,\lambda_2,\Lambda}}\int_{\substack{|a-a_\ast|\leq l_{clock}\\|\Phi_t(a_*)-\Phi_t(a)|\leq L\\
    a\in E_{\lambda_1,\lambda_2,\Lambda}}}\frac{\tau_1|\overline{\omega}_{0,E}(\eta_3)-\omega_0(a_*)|}{|\Phi_t(a_*)-\Phi_t(a)|^3}\,da\leq B_{\parallel}(t).   
\end{align*}
Thus,
\begin{equation}\label{Clock longitudinal Driver}
    |\mathcal{D}_{clock,\parallel}(t)|\leq B_{\parallel}(t)
\end{equation}
We move on to the off-clock contribution. Since $\|F(a_\ast,t)\|_{op}\leq\tau_1$, we have

\begin{align*}
 |I_{\lambda_1,\lambda_2,\Lambda,off-clock}|&\lesssim \int_{\substack{|a-a_\ast|>l_{clock}\\|\Phi_t(a_*)-\Phi_t(a)|\leq L\\
    a\in E_{\lambda_1,\lambda_2,\Lambda}}}\frac{\tau_1|\omega_{0}(a)-\omega_0(a_*)|}{|\Phi_t(a_*)-\Phi_t(a)|^3}\,da   
\end{align*}
The off-clock coherence condition now guarantees that
\begin{align*}
 \sum_{\substack{\lambda_1,\lambda_2,\Lambda}}|I_{\lambda_1,\lambda_2,\Lambda,off-clock}|&\lesssim \sum_{\substack{\lambda_1,\lambda_2,\Lambda}}\int_{\substack{|a-a_\ast|> l_{clock}\\|\Phi_t(a_*)-\Phi_t(a)|\leq L\\
    a\in E_{\lambda_1,\lambda_2,\Lambda}}}\frac{\tau_1|\omega_{0}(a)-\omega_0(a_*)|}{|\Phi_t(a_*)-\Phi_t(a)|^3}\,da\\
    &\lesssim \int_{\substack{|a-a_\ast|> l_{clock}\\|\Phi_t(a_*)-\Phi_t(a)|\leq L}}\frac{\tau_1|\omega_{0}(a)-\omega_0(a_*)|}{|\Phi_t(a_*)-\Phi_t(a)|^3}\,da\leq B_{R}(t).   
\end{align*}
Thus,
\begin{equation}\label{Off-clock Driver}
  |\mathcal{D}_{off-clock}(t)|\leq B_{R}(t). 
\end{equation}
For the $II$ contribution, we shall use the averaged Dini condition. To be more precise, we have

\begin{align*}
    |II_{\lambda_1,\lambda_2,\Lambda}|&\lesssim \int_{\substack{E_{\lambda_1,\lambda_2,\Lambda}\\
    |\Phi_t(a)-\Phi_t(a_\ast)|\leq L}}\frac{
\left|
\bigl[F(a,t)-F(a_*,t)\bigr]\omega_0(a)
\right|
}{
\left|
\Phi_t(a)-\Phi_t(a_\ast)
\right|^3
}\,da,
\end{align*}
hence
\begin{align*}
    \sum_{\lambda_1,\lambda_2,\Lambda}|II_{\lambda_1,\lambda_2,\Lambda}|&\lesssim \sum_{\lambda_1,\lambda_2,\Lambda}\int_{\substack{E_{\lambda_1,\lambda_2,\Lambda}\\
    |\Phi_t(a)-\Phi_t(a_\ast)|\leq L}}\frac{\left|
\bigl[F(a,t)-F(a_*,t)\bigr]\omega_0(a)
\right|
}{
\left|
\Phi_t(a)-\Phi_t(a_\ast)
\right|^3
}\,da\lesssim B_F(t).
\end{align*}
Therefore,
\begin{equation}\label{Dini Driver}
|\mathcal{D}_{Dini}(t)|\leq B_F(t).
\end{equation}
We now analyze the transverse clock contributions.
\subsubsection{The off-cone region}

In this region, the distance satisfies $d\simeq \nu\Lambda(1+\lambda^2)^{\frac{1}{2}}$, where $\lambda^2=\lambda_1^2+\lambda_2^2$, and the kernel satisfies the bound
\begin{align*}
    |K_e(x_*-y)|\lesssim \nu^{-3}\Lambda^{-3}(1+\lambda^2)^{-\frac{3}{2}}
\end{align*}
 Here, $|\eta_\perp|=\sqrt{\eta_1^2+\eta_2^2}\lesssim \frac{\lambda\nu\Lambda}{\tau_2}$. We also have $\|F(a_*,t)\|_{op}\leq\tau_1$. 

 By the integrated transverse cusp/tail condition, it follows that

\begin{align*}
    |I_{\lambda_1,\lambda_2,\Lambda,off-cone,clock,\perp}|&\lesssim \tau_1\nu^{-3}\Lambda^{-3}(1+\lambda^2)^{-\frac{3}{2}}\int_{E^{clock}_{\lambda_1,\lambda_2,\Lambda}}|\omega_0(a)-\overline{\omega}_{0,E}(\eta_3)|\,d\eta\\
    &\lesssim \tau_1\nu^{-3}\Lambda^{-3}(1+\lambda^2)^{-\frac{3}{2}}M_\alpha b_\Lambda(a_*,t)\left(\frac{\lambda\nu\Lambda}{\tau_2}\right)^\alpha\frac{\lambda_1\lambda_2\nu^2\Lambda^2}{\tau_1\tau_2}\Lambda\\
    &\lesssim M_\alpha b_\Lambda(a_*,t) \nu^{\alpha-1}\Lambda^{\alpha}\tau_2^{-1-\alpha}\frac{\lambda^\alpha\lambda_1\lambda_2}{(1+\lambda^2)^{\frac{3}{2}}}.
\end{align*}

 We know that $\tau_1\tau_2\nu=1$ and $\tau_1=\kappa\tau_2$. Then, $\tau_2=\nu^{-\frac{1}{2}}\kappa^{-\frac{1}{2}}$ and $\tau_1=\nu^{-\frac{1}{2}}\kappa^{\frac{1}{2}}$. This implies that

\begin{align*}
    |I_{\lambda_1,\lambda_2,\Lambda,off-cone,\perp}|&\lesssim M_\alpha b_\Lambda(a_*,t)\Lambda^{\alpha}\frac{\lambda^{\alpha}\lambda_1\lambda_2}{(1+\lambda^2)^{\frac{3}{2}}}\nu^{\frac{3\alpha-1}{2}}\kappa^{\frac{\alpha+1}{2}}
\end{align*}

We have
\begin{align*}
    \sum_{\lambda\ll 1}\frac{\lambda^{\alpha}\lambda_1\lambda_2}{(1+\lambda^2)^{\frac{3}{2}}}+\sum_{\lambda\gg 1}\frac{\lambda^{\alpha}\lambda_1\lambda_2}{(1+\lambda^2)^{\frac{3}{2}}}&\lesssim \sum_{\lambda\ll 1}\lambda^{\alpha}\lambda_1\lambda_2+\sum_{\lambda\gg 1}\lambda^{\alpha-3}\lambda_1\lambda_2\\
    &\lesssim \sum_{\lambda_1,\lambda_2\ll 1}(\lambda_1^{\alpha}+\lambda_2^{\alpha})\lambda_1\lambda_2+\sum_{\substack{n\geq 0\\
    n_1,n_2\in\mathbb{Z}\\
    \max(n_1,n_2)=n}}2^{n_1+n_2}2^{n(\alpha-3)}\\
    &\lesssim 1+\sum_{\substack{n\geq 0}}2^{2n}2^{n(\alpha-3)}<\infty
\end{align*}
In the last term, we have used the fact that $\displaystyle \sum_{\substack{
    n_1,n_2\in\mathbb{Z}\\
    \max(n_1,n_2)=n}}2^{n_1+n_2}=32^{2n}$.

We then have
\begin{align*}
\sum_{\lambda\not\simeq 1}|I_{\lambda_1,\lambda_2,\Lambda,off-cone,clock,\perp}|&\lesssim M_\alpha b_\Lambda(a_*,t)\Lambda^{\alpha}\nu^{\frac{3\alpha-1}{2}}\kappa^{\frac{\alpha+1}{2}}
\end{align*}

We now use the tail condition to sum over $\Lambda$, as well as the uniformity condition with respect to $a_*$, along with the fact that $E^{clock}_{\lambda_1,\lambda_2,\Lambda}$ is empty for $\Lambda\gg l_{clock}$ to deduce that
\begin{align*}
\sum_{\Lambda}\sum_{\lambda\not\simeq 1}|I_{\lambda_1,\lambda_2,\Lambda,off-cone,clock,\perp}|&\lesssim \sum_{\Lambda\lesssim l_{clock}}M_\alpha b_\Lambda(a_*,t)\Lambda^{\alpha}\nu^{\frac{3\alpha-1}{2}}\kappa^{\frac{\alpha+1}{2}}\lesssim M_\alpha\mathcal{T}_\alpha\nu^{\frac{3\alpha-1}{2}}\kappa^{\frac{\alpha+1}{2}}\\
&\lesssim C_{T_*}\nu^{\frac{3\alpha-1}{2}},
\end{align*}
where the last inequality follows from the boundedness of $\kappa$ (the bounded transverse distortion condition).

Thus,
\begin{equation}\label{Off-cone transverse Driver}
    |\mathcal{D}_{off-cone,clock,\perp}(t)|\leq C_{T_\ast}\nu^{\frac{3\alpha-1}{2}}(t).
\end{equation}

\subsubsection{The cone region}

We now move our attention to the cone, which is defined by the condition $\lambda\simeq 1$. Here, the Eulerian distance satisfies $d(a,t)\simeq \nu\Lambda$, and the kernel satisfies the bound
\begin{align*}
    |K_e(x_*-y)|\lesssim \nu^{-3}\Lambda^{-3}.
\end{align*}
As before, $|\eta_\perp|=\sqrt{\eta_1^2+\eta_2^2}\lesssim \frac{\nu\Lambda}{\tau_2}$. We also have $\|F(a_*,t)\|_{op}\leq\tau_1$. By the integrated transverse cusp/tail condition, it follows that

\begin{align*}
    |I_{\lambda_1,\lambda_2,\Lambda,cone,\perp}|&\lesssim \tau_1\nu^{-3}\Lambda^{-3}\int_{E^{clock}_{\lambda_1,\lambda_2,\Lambda}}|\omega_0(a)-\overline{\omega}_{0,E}(\eta_3)|\,da\\
    &\lesssim \tau_1\nu^{-3}\Lambda^{-3}M_\alpha b_\Lambda(a_*,t)\left(\frac{\nu\Lambda}{\tau_2}\right)^\alpha\frac{\lambda_1\lambda_2\nu^2\Lambda^2}{\tau_1\tau_2}\Lambda\\
    &\lesssim M_\alpha b_\Lambda(a_*,t) \nu^{\alpha-1}\Lambda^{\alpha}\tau_2^{-1-\alpha}\lambda_1\lambda_2.
\end{align*}

 We know that $\tau_1\tau_2\nu=1$ and $\tau_1=\kappa\tau_2$. Then, $\tau_2=\nu^{-\frac{1}{2}}\kappa^{-\frac{1}{2}}$ and $\tau_1=\nu^{-\frac{1}{2}}\kappa^{\frac{1}{2}}$. This implies that

\begin{align*}
    |I_{\lambda_1,\lambda_2,\Lambda,cone,\perp}|&\lesssim M_\alpha b_\Lambda(a_*,t)\Lambda^{\alpha}\frac{\lambda^{\alpha}\lambda_1\lambda_2}{(1+\lambda^2)^{\frac{3}{2}}}\nu^{\frac{3\alpha-1}{2}}\kappa^{\frac{\alpha+1}{2}}
\end{align*}

We now use the tail condition to sum over $\Lambda$, and we have
\begin{align*}
\sum_{\substack{\Lambda\\\lambda\simeq 1}}|I_{\lambda_1,\lambda_2,\Lambda,cone,\perp}|&\lesssim \sum_{\Lambda}M_\alpha b_\Lambda(a_*,t)\Lambda^{\alpha}\nu^{\frac{3\alpha-1}{2}}\kappa^{\frac{\alpha+1}{2}}\lesssim M_\alpha\mathcal{T}_\alpha\nu^{\frac{3\alpha-1}{2}}\kappa^{\frac{\alpha+1}{2}}\\
&\lesssim C_{T_*}\nu(t)^{\frac{3\alpha-1}{2}},
\end{align*}
where the last inequality follows from the boundedness of $\kappa$ (the bounded transverse distortion condition).

\begin{equation}\label{Cone transverse Driver}
    |\mathcal{D}_{cone,clock,\perp}(t)|\leq C_{T_\ast}\nu^{\frac{3\alpha-1}{2}}(t).
\end{equation}

We now move our attention to the far labels, which consists of the points $y$ with $|x_*-y|\geq \frac{L}{2}$. We recall that

\begin{align*}
    \mathcal{D}_{\rm far}&=\int \chi_L(x_*-y)K_e(x_*-y)\omega(y,t)\,dy.
\end{align*}

By integrating by parts, we have
\begin{align*}
 \mathcal{D}_{far}&=-\int \nabla_y\times(\chi_L(x_*-y)K_e(x_*-y))u(y,t)\,dy.
\end{align*}
We have the bound
\begin{align*}
 |\nabla_y\times (\chi_L(x_*-y)K_e(x_*-y))|&\lesssim |x_*-y|^{-4}+L^{-1}|x_*-y|^{-3}\mathbf{1}_{\frac{L}{2}\leq|x_*-y|\leq L},   
\end{align*}
hence by Cauchy-Schwarz's inequality,
\begin{align*}
    \Big|&\int \chi_L(x_*-y)\nabla_y\times K_e(x_*-y)u(y,t)\,dy\Big|\lesssim \int_{|x_*-y|\geq\frac{L}{2}}|x_*-y|^{-4}|u(y,t)|\,dy\\
    &+L^{-1}\int_{L\geq |x_*-y|\geq\frac{L}{2}}|x_*-y|^{-3}|u(y,t)|\,dy\\
    &\lesssim \int_{|y|\geq\frac{L}{2}}|y|^{-4}|u(x_*-y,t)|\,dy+L^{-1}\int_{L\geq |y|\geq\frac{L}{2}}|y|^{-3}|u(x_*-y,t)|\,dy\\
    &\lesssim (\||y|^{-4}\|_{L_y^2(|y|\geq\frac{L}{2})}+L^{-1}\||y|^{-3}\|_{L_y^2(L\geq |y|\geq\frac{L}{2})})\|u(t)\|_{L^2}\\
    &\lesssim L^{-\frac{5}{2}}\|u_0\|_{L^2}
\end{align*}

Here, we have used the energy conservation property $\|u(t)\|_{L^2}=\|u_0\|_{L^2}$.

These lead to the bound
\begin{equation}\label{Far Driver Estimate}
  |\mathcal{D}_{far}| \lesssim C_{T_*}. 
\end{equation}
Estimates \eqref{Clock longitudinal Driver},\eqref{Cone transverse Driver}, \eqref{Off-cone transverse Driver},\eqref{Off-clock Driver}, and \eqref{Far Driver Estimate} now imply that 
\begin{align*}
    |\mathcal{D}_*(t)|&\leq |\mathcal{D}_{\rm cone,clock,\perp}(t)|+|\mathcal{D}_{\rm off-cone,clock,\perp}(t)|+|\mathcal{D}_{\rm clock,\parallel}(t)|\\
    &+|\mathcal{D}_{off-clock}(t)|+|\mathcal{D}_{Dini}(t)|+|\mathcal{D}_{\rm far}(t)|\\
    &\leq C_T+B_{\parallel}(t)+B_{R}(t)+B_{F}(t)+C_{T_*}\nu(t)^{\frac{3\alpha-1}{2}}\\
    &\leq B(t)+C_{T_*}\nu(t)^{\frac{3\alpha-1}{2}}, 
\end{align*}
where $B(t)$ is positive and in $L^1([0,T))$, as desired. Here, we have absorbed the far-field contribution $C_{T_*}$ into $B(t)$.

\section{The proof of Theorem \ref{Main Result 2}}\label{Section 6}
This is very similar to the one of Theorem \ref{Main Result}, and is in fact a simplification. We shall omit some details, while only focusing on the aforementioned simplifications.

Let $a_*=(0,\Theta,Z)$ be an arbitrary axis label.

We recall that the deformation matrix is given by
 \begin{align*}
  F_*(t):&=\begin{pmatrix}
\partial_R\varphi^r & 0 & \partial_Z\varphi^r\\
0 & \frac{\varphi^r}{R} & 0\\
\partial_R\varphi^z & 0 & \partial_Z\varphi^z
      \end{pmatrix}   
 \end{align*}

By axisymmetry, we have the expansions
\begin{align*}
  \varphi^r(R,Z,t)&=a(Z,t)R+o(R)\\
  \varphi^z(R,Z,t)&=b(Z,t)+o(R).
\end{align*}

It follows that $\partial_R\varphi^r(0,Z,t)=a(Z,t)$, $\partial_R\varphi^z(0,Z,t)=0$,$\partial_Z\varphi^r(0,Z,t)=0$, and that
\begin{align*}
    \lim_{\substack{R\rightarrow 0}}\frac{\varphi^r(R,Z,t)}{R}&=a(Z,t).
\end{align*}
Therefore, in a suitable orthogonal frame, the deformation matrix for the axis label $a_*$ is given by
\begin{align*}
  F_{\rm axis}(t):&=\begin{pmatrix}
a & 0 & 0\\
0 & a & 0\\
0 & 0 & \partial_Z\varphi^z
      \end{pmatrix}   
 \end{align*}

The reduced meridional Jacobian is given by
\begin{align*}
    J(R,Z,t)=\det D_{R,Z}(\varphi^r,\varphi^z)=\frac{R}{\varphi^r(R,Z,t)}.
\end{align*}
By taking the limit as $R\rightarrow 0$, we get that
\begin{align*}
    J_*(t)=J(0,Z,t)=\frac{1}{a(Z,t)}
\end{align*}

By incompressibility, $\det F_{\rm axis}=1$, hence we immediately get that $\partial_Z\varphi^z=a^{-2}=J^2$.

In particular, this shows that the singular values of $F_{\rm axis}$ are given by $\tau_1=\tau_2=J_*^{-1}$ and that the smallest singular value is given by $\nu=J_*^2$. In particular, $\kappa(t)=\frac{\tau_1(t)}{\tau_2(t)}=1$.

From a geometrical point of view, this tells us that in the regime $J_*<1$, which is the only relevant one for collapse, the radial/azimuthal directions expand like $J_*^{-1}$, while the axial directions contract like $J_*^2$.

The near-field axis comparability condition implies that $r\sim J_*^{-1}R$ and
\begin{align*}
|\Phi_t(a)-\Phi_t(a_\ast)|\simeq |D_a\Phi_t(a_\ast)(a-a_\ast)|\simeq (J_\ast^{-2}R^2+J_\ast^4(\Delta Z)^2)^{\frac{1}{2}}.
\end{align*}
If we set $\displaystyle \lambda=\frac{R}{J_\ast^3|\Delta Z|}$, we get that 
\begin{align*}
|\Phi_t(a)-\Phi_t(a_\ast)|\simeq J_\ast^2|\Delta Z|\sqrt{1+\lambda^2}.
\end{align*}
For small angles this is compatible with Shkoller's drift law
\begin{align*}
    \frac{r}{z-z_\ast}=J_*^{-3}\frac{R}{Z-Z_\ast}.
\end{align*}

We note that in this case, $E_{\nu(t)}$ is spanned by the axis unit vector, and that the only candidate for the driver is
\begin{align*}
    W_*(t)&:=e_z\cdot Se_z=\partial_zu^z(0,\varphi^z(0,Z,t),t).
\end{align*}

In this case, the desired estimate takes the form
\begin{align*}
    |W_*(t)|\leq B(t)+C_TJ_*(t)^{3\alpha-1},
\end{align*}
where the hypotheses of Theorem \ref{Main Result 2} guarantee that the constant $C_T$ and the function $B$ are independent of the on-axis label $a_*$. Here, just like in the general label case, $B(t)$ is a positive function in $L^1([0,T))$ which controls the far-field contribution, and generates a lower-order term.

From Shkoller \cite{Shkoller}, we know that
\begin{align*}
    \frac{d}{dt}J_*(t)=\frac{1}{2}J_*(t)W_*(t)
\end{align*}
We immediately get that
\begin{align*}
  \frac{d}{dt}J_*(t)&\geq -B(t)J_*(t)-C_T J_*(t)^{3\alpha}.  
\end{align*}

When $3\alpha=1$, then the inequality becomes
\begin{align*}
  \frac{d}{dt}J_*(t)&\geq -(B(t)+C_{T_*})J_*(t).  
\end{align*}
A direct application of Gr\"onwall's inequality, together with the fact that $J_*(0)=1$, lead to the estimate
\begin{align*}
    J_*(t)\geq e^{-\int_0^tB(\tau)\,d\tau-C_{T_*}t},
\end{align*}
for every $t\in[0,T_*)$. 

When $\alpha>\frac{1}{3}$, $3\alpha>1$, hence the comparison ODE becomes
\begin{align*}
    \dot{f}(t)&=-B(t)f(t)-C_{T_*} f(t)^{3\alpha},
\end{align*}
with initial data $f(0)=1$.

By dividing by $f(t)^{3\alpha}$ and setting $\displaystyle g(t):=\frac{f(t)^{1-3\alpha}}{1-3\alpha}$, we obtain the ODE

\begin{align*}
    \dot{g}(t)&=-(1-3\alpha)B(t)g(t)-C_{T_*},
\end{align*}
and another direct application of Gr\"onwall's inequality allows us to obtain a uniform bound on $-g$ on $[0,T_\ast)$. Since the exponent of $f$ in the definition of $g$ is negative, this implies that $f$ cannot hit zero in finite time. Thus, $J_*(t)$ cannot hit zero in finite time either.

In either case, we have proved that $J_*(t)\geq c_T>0$, where $c_T$ is a constant independent of the on-axis label $a_*$. 

We now return to the proof of the estimate. The only relevant case for us is $J_*(t)\leq1$. Let $a=a_*+\eta$. By the Biot-Savart Law \ref{Biot-Savart},
\begin{align*}
   W_*(t)&=c_0\int_{\mathbb{R}^3}\frac{r_y(z_*-z_y)}{|x_*-y|^5}\omega^\theta(y)\,dy\\
   &=c_0\int_{\mathbb{R}^3_a}\frac{r(\Phi_t(a))(z_*-z(\Phi_t(a)))}{|x_*-\Phi_t(a)|^5}\omega^\theta(\Phi_t(a))\,da,
\end{align*}
where $\omega^\theta=\partial_zu^r-\partial_ru^z$, and $c_0$ is a constant.

We note that since $\omega_0(a_*)=0$, we must also have $\omega(x_*,t)=0$. Together with the fact that $\omega_0\in C^{0,\alpha}$, this implies via the H\" older continuity modulus that the integral is properly defined, with no principal value needed.

In what follows, we shall only focus on the near-field labels, which as before, are defined as the ones satisfying $|\Phi_t(a)-\Phi_t(a_*)|\leq L$.

By the near field axis condition, we have

\begin{align*}
    r(\Phi_t(a))&\simeq\frac{R}{J_*}
\end{align*}

Let \begin{align*}
    \Lambda&=|\Delta Z|,\\
    \lambda&=\frac{R}{J_*^3\Lambda}.
\end{align*}

On the dyadic block $E_{\lambda,\Lambda}$, where
\begin{align*}
    E_{\lambda,\Lambda}=\left\{|\eta_3|\simeq\Lambda,|\eta_\perp|\simeq\lambda J^3_*\Lambda\right\},
\end{align*}
we have $\displaystyle |x_*-\varphi(a,t)|\simeq J^2_*\Lambda(1+\lambda^2)^{\frac{1}{2}}$ (here, $\lambda,\Lambda\in 2^{\mathbb{Z}}$). In particular, it follows that $|z_\ast-z(\Phi_t(a))|\lesssim J^2_*\Lambda(1+\lambda^2)^{\frac{1}{2}}$.

We then define the cone as the region where $\lambda\simeq 1$, and the off-cone region as the one where $\lambda\not \simeq 1$.

We first analyze the changes in the near labels.

In this dyadic block, the kernel in the integral satisfies
\begin{align*}
    |K_{W*,\Lambda,\lambda}|&=\left|\frac{r(\Phi_t(a))(z_*-z(\Phi_t(a)))}{|x_*-\Phi_t(a)|^5}\right|\lesssim \frac{\lambda J^2_*\Lambda J^2_*\Lambda(1+\lambda^2)^{\frac{1}{2}}}{J^{10}_*\Lambda^5(1+\lambda^2)^{\frac{5}{2}}}=\frac{\lambda}{J^{6}_*\Lambda^3(1+\lambda^2)^{2}}
\end{align*}

By the near field axis condition, since we are in the collapsing regime $J_*(t)<1$, which is the only one relevant for us, we have $\|D\Phi_t(a)\|_{op}\simeq \|D\Phi_t(a_*)\|_{op}=J_*(t)^{-1}$. 
We now pass to cylindrical coordinates in the Biot-Savart Law
\begin{align*}
   W_*(t)&=-\frac{3}{2}\int_{0}^\infty\int_{\mathbb{R}}\frac{r(\Phi_t(a))(z_*-z(\Phi_t(a)))}{|x_*-\Phi_t(a)|^5}\omega^\theta(\Phi_t(a))R\,dR\,dZ\\
   &=-\frac{3}{2}\int_{0}^\infty\int_{\mathbb{R}}\frac{r(\Phi_t(a))(z_*-z(\Phi_t(a)))}{|x_*-\Phi_t(a)|^5}\langle(D_a\Phi_t(a)\omega_0(a)),e^\theta(\Theta)\rangle R\,dR\,dZ
\end{align*}
in a dyadic region where $\lambda$ and $\Lambda$ are fixed. In particular, the integration measure $R\,dR\,dZ$ satisfies $\displaystyle \int_{E_{\lambda,\Lambda}}R\,dR\,dZ\simeq\lambda^2 J_*^6\Lambda^3$.

We further split
\begin{align*}
   W_*(t)&= -\frac{3}{2}\int_{\substack{|\Phi_t(a)-\Phi_t(a_\ast)|\leq L\\|a-a_\ast|\leq l_{clock}}}\frac{r(\Phi_t(a))(z_*-z(\Phi_t(a)))}{|x_*-\Phi_t(a)|^5}(D_a\Phi_t(a)\omega_0(a))R\,dR\,dZ\\
   &-\frac{3}{2}\int_{\substack{|\Phi_t(a)-\Phi_t(a_\ast)|\leq L\\|a-a_\ast|> l_{clock}}}\frac{r(\Phi_t(a))(z_*-z(\Phi_t(a)))}{|x_*-\Phi_t(a)|^5}(D_a\Phi_t(a)\omega_0(a))R\,dR\,dZ\\
   &-\frac{3}{2}\int_{\substack{|\Phi_t(a)-\Phi_t(a_\ast)|> L}}\frac{r(\Phi_t(a))(z_*-z(\Phi_t(a)))}{|x_*-\Phi_t(a)|^5}(D_a\Phi_t(a)\omega_0(a))R\,dR\,dZ\\
   &:=W_{\ast,near,clock}+W_{\ast,near,off-clock}+W_{\ast,far}
\end{align*}
Let $W_{*,\lambda,\Lambda,near,clock}$ be the contribution of the integral defining $W_{\ast,near,clock}$ in the region $E_{\lambda,\Lambda}$.

By the integrated cusp/tail estimate, we obtain 
\begin{align*}
   |W_{*,\lambda,\Lambda, near,clock}|&\lesssim C_T\frac{\lambda J_*^{-1}}{J^{6}_*\Lambda^3(1+\lambda^2)^{2}}\int_{\substack{E^{clock}_{\lambda,\Lambda}(a_\ast,t)\\
   |\Phi_t(a)-\Phi_t(a_\ast)|\leq L\\|a-a_\ast|\leq l_{clock}}}  |\omega_0(a)|\,da\\
   &\lesssim C_T\frac{\lambda J_*^{-1}}{J^{6}_*\Lambda^3(1+\lambda^2)^{2}} \lambda^{\alpha}J^{3\alpha}_*b_\Lambda(a_*,t)\Lambda^{\alpha}\lambda^2 J_*^6\Lambda^3\\
   &\lesssim C_T\frac{\lambda^{\alpha+3}}{(1+\lambda^2)^{2}} J^{3\alpha-1}_* \Lambda^{\alpha}b_\Lambda(a_*,t)
\end{align*}
Since for $\Lambda\gg l_{clock}$, the set $E^{clock}_{\lambda,\Lambda}$ is empty, we have 
\begin{align*}
 |W_{*,near,clock}| &\lesssim \sum_{\Lambda\lesssim l_{clock}}\sum_{\lambda}\frac{\lambda^{\alpha+3}}{(1+\lambda^2)^{2}} J^{3\alpha-1}_* \Lambda^{\alpha}b_\Lambda(a_*,t)\lesssim J^{3\alpha-1}_*\sum_{\Lambda\lesssim l_{clock}}\Lambda^{\alpha}b_\Lambda(a_*,t)\sum_{\lambda}\frac{\lambda^{\alpha+3}}{(1+\lambda^2)^{2}}\\
 &\lesssim J^{3\alpha-1}_*\sum_{\Lambda\lesssim l_{clock}}\Lambda^{\alpha}b_\Lambda(a_*,t)\left(\sum_{\lambda<1}\frac{\lambda^{\alpha+3}}{(1+\lambda^2)^{2}}+\sum_{\lambda\geq 1}\frac{\lambda^{\alpha+3}}{(1+\lambda^2)^{2}}\right)\\
 &\lesssim J^{3\alpha-1}_*\sum_{\Lambda\lesssim l_{clock}}\Lambda^{\alpha}b_\Lambda(a_*,t)\left(\sum_{\lambda<1}\lambda^{\alpha+3}+\sum_{\lambda\geq 1}\lambda^{\alpha-1}\right)\\
 &\lesssim C_{T_*}J^{3\alpha-1}_*
\end{align*}

We now analyze $W_{\ast,near,off-clock}$. By the off-clock coherence condition, we have
\begin{align*}
 |W_{*,near,off-clock}| &\lesssim \int_{\substack{|\Phi_t(a)-\Phi_t(a_\ast)|\leq L\\|a-a_\ast|> l_{clock}}}\frac{r(\Phi_t(a))|z_*-z(\Phi_t(a))|}{|x_*-\Phi_t(a)|^5}|\omega(t,\Phi_t(a))|R\,dR\,dZ\\
 &\lesssim \int_{\substack{|\Phi_t(a)-\Phi_t(a_\ast)|\leq L\\|a-a_\ast|> l_{clock}}}\frac{r(\Phi_t(a))|z_*-z(\Phi_t(a))|}{|x_*-\Phi_t(a)|^5}|\omega(t,\Phi_t(a))|\,da\\
 &\lesssim B_R^{axis}(t).
\end{align*}

For the far-field labels, the previous energy argument works in exactly the same way, yielding an estimate similar to \eqref{Far Driver Estimate}, and we omit the proof. We note that as we have already claimed, the averaged Dini coherence condition was no longer necessary for on axis labels.
\section*{Acknowledgements} 
The author would like to thank Steve Shkoller for his very interesting talk in the Yale Analysis Seminar, for carefully reading the final draft and for providing very helpful comments, and Wilhelm Schlag for suggesting this project and for very helpful discussions. 

The author would also like to thank the Yale Mathematics Department for its hospitality during the 2025-2026 academic year.

The author used ChatGPT 5.5 Pro and ChatGPT 5.6 Sol for exploratory discussion and editorial feedback during the preparation of this manuscript. All mathematical statements, proofs, and final formulations are humanly written and are the author’s responsibility.

Theorems 1.1 and 1.2, Propositions 3.1, 3.2, and 4.2, as well as Lemma 4.1 have been formalized in Lean 4, and the corresponding Lean code is available in the author's  \href{https://github.com/ovidiuavadanei1997/Lean-Formalizations/tree/b139da95bb02aafe5a0f150310e07257c156359e/Conditional%20clock%20barrier%20for%20the%203-d%20axisymmetrical%20Euler%20with%20no%20swirl}{GitHub repository}. For the formalization of the results, the author used Codex (GPT-6 Astra) and supervised the entire process, while the verification was carried out using Lean's kernel. The exact Lean-formalized statements can be found in the \href{https://github.com/ovidiuavadanei1997/Lean-Formalizations/blob/b139da95bb02aafe5a0f150310e07257c156359e/Conditional%20clock%20barrier%20for%20the%203-d%20axisymmetrical%20Euler%20with%20no%20swirl/verification/Statements.lean}{Statements.lean file}. The repository also contains a \href{https://github.com/ovidiuavadanei1997/Lean-Formalizations/blob/b139da95bb02aafe5a0f150310e07257c156359e/Conditional%20clock%20barrier%20for%20the%203-d%20axisymmetrical%20Euler%20with%20no%20swirl/THEOREM_GUIDE.md}{guide} matching each result to its named Lean declaration and implementation. 
\bibliography{bibliography.bib}
\bibliographystyle{plain}
\end{document}